\documentclass[reqno,tbtags]{amsart}%

\usepackage{amssymb}
\usepackage[shortlabels]{enumitem}
\usepackage{mathrsfs}
\newtheorem{thm}{Theorem}[section]
\newtheorem{lem}[thm]{Lemma}
\newtheorem{prop}[thm]{Proposition}
\newtheorem{cor}[thm]{Corollary}

\theoremstyle{definition}
\newtheorem{dfn}[thm]{Definition}

\theoremstyle{remark}
\newtheorem{rem}[thm]{Remark}
\newtheorem{exm}[thm]{Example}

\numberwithin{equation}{section}

\newcommand{\dN}{\mathbb{N}}
\newcommand{\dZ}{\mathbb{Z}}
\newcommand{\dR}{\mathbb{R}}
\newcommand{\dC}{\mathbb{C}}
\newcommand{\dP}{\mathbb{P}}

\newcommand{\ii}{\mathrm{i}}
\newcommand{\scrL}{\mathscr{L}}
\newcommand{\trn}{\mathrm{T}}

\DeclareMathOperator{\vvec}{vec}
\DeclareMathOperator{\Lin}{Lin}
\DeclareMathOperator{\tr}{tr}
\DeclareMathOperator{\rank}{rank}
\DeclareMathOperator{\cone}{cone}
\newcommand{\ov}{\overline}
\newcommand{\clo}{\overline}
\newcommand{\dif}{\mathrm{d}}
\DeclareMathOperator{\ran}{ran}
\DeclareMathOperator{\Car}{Car}
\DeclareMathOperator{\car}{car}
\DeclareMathOperator{\ats}{at}

\newcommand{\cE}{\mathcal{E}}

\newcommand{\cH}{\mathcal{H}}

\newcommand{\cL}{\mathcal{L}}
\newcommand{\cM}{\mathcal{M}}
\newcommand{\cN}{\mathcal{N}}

\newcommand{\cU}{\mathcal{U}}
\newcommand{\cV}{\mathcal{V}}
\newcommand{\cW}{\mathcal{W}}
\newcommand{\cX}{\mathcal{X}}
\newcommand{\cY}{\mathcal{Y}}

\newcommand{\cHp}{\cH_{q,\succeq}}

\newcommand{\fC}{\mathfrak{C}}
\newcommand{\fL}{\mathfrak{L}}

\newcommand{\fU}{\mathfrak{U}}
\newcommand{\fX}{\mathfrak{X}}
\newcommand{\fY}{\mathfrak{Y}}

\newcommand{\bU}{\mathbf{U}}
\newcommand{\bV}{\mathbf{V}}
\newcommand{\bW}{\mathbf{W}}

\newcommand{\rnd}{\Phi}
\newcommand{\rnde}{\phi}
\DeclareMathOperator{\atnr}{nr}

\begin{document}
\title{On the Truncated Matricial Moment Problem. I}

\author{Conrad M\"adler}
\address{University of Leipzig,
Mathematical Institute,
Augustusplatz~10,
04109~Leipzig,
Germany}
\email{maedler@math.uni-leipzig.de}

\author{Konrad Schm\"udgen}
\address{University of Leipzig,
Mathematical Institute,
Augustusplatz~10,
04109~Leipzig,
Germany}
\email{schmuedgen@math.uni-leipzig.de}


\begin{abstract}
This paper is about the general truncated matrix-valued moment problem.
Let $\cH_q$ denote the complex Hermitian $q\times q$-matrices, $q\in \dN$.
Suppose that $(\cX,\fX)$ is a measurable space and $\cE$ is a finite-dimensional vector space of measurable mappings of $\cX$ into $\cH_q$.
A linear functional $\Lambda$ on $\cE$ is called a moment functional if there exists a positive $\cH_q$-valued measure $\mu$ on $(\cX,\fX)$  such that $\Lambda(F)=\int_\cX \langle F,\dif\mu\rangle$ for $F\in \cE$.

We prove a matricial version of the Richter--Tchakaloff theorem which states that each moment functional on $\cE$ has a finitely atomic representing measure.
It is shown that strictly positive linear functionals on $\cE$ are moment functionals.
For a moment functional $\Lambda$, we study the set of atoms $\cW(\Lambda)$ and the  Carath\'eodory numbers $\Car (\Lambda)$, $\car(\Lambda)$ and we define and investigate the core set $\cV(\Lambda)$.
A main result of the paper is the equality $\cW(\Lambda)=\cV(\Lambda)$.
\end{abstract}

\maketitle

\section{Introduction}
The aim of this paper is to study the general matricial truncated moment problem.
More precisely, we are concerning with the following setup:
Suppose $(\cX,\fX)$ is a measurable space such that $\fX$ contains all singleton sets $\{x\}\in \fX$ for  $x\in \cX$.
Fix a number $q\in \dN$.
Let $\cH_q$ denote the complex Hermitian $q\times q$-matrices and $\cM_q(\cX,\fX)$ the set of positive $\cH_q$-valued measures on $(\cX,\fX)$.

Suppose that $\cE$ is a finite-dimensional real vector space of measurable mappings $F\colon\cX\to \cH_q$.
A linear functional $\Lambda$ on $\cE$ is called a \emph{moment functional} if there exists a measure $\mu\in\cM_q(\cX,\fX)$ such that
\begin{equation}\label{mf}
    \Lambda(F)
    =\int_\cX \langle F(x),\dif\mu\rangle\quad\text{for all }F\in \cE.
\end{equation}
An important class of measures of $\cM_q(\cX,\fX)$ are finitely atomic measures of the form $\mu=\sum_{j=1}^k \delta_{x_j} M_j$, where $x_1,\dotsc,x_k\in \cX$ and $M_1,\dotsc,M_k$ are non-negative matrices of $\cH_q$.
Then \eqref{mf} reads as
\begin{equation}\label{mf1}
    \Lambda (F)
    =\sum_{j=1}^k\, \tr(F(x_j)M_j)\quad\text{for all }F\in \cE.
\end{equation}
Further, let $E$ be a finite-dimensional vector space of real-valued measurable functions on $(\cX,\fX)$.
A linear mapping $L$ of $E$ into $\cH_q$ is called a \emph{matrix moment functional} if there is a measure $\mu\in\cM_q(\cX,\fX)$ such that
\begin{equation}\label{mf2}
    L(f)
    =\int_\cX f(x) \,  \dif\mu\quad\text{for all }f\in E.
\end{equation}
This paper and its successor are devoted to a systematic study of moment functionals \eqref{mf} on $\cE$ and of matrix moment functionals \eqref{mf2} on $E$.
Though there exist a well-developed theory of the scalar truncated moment problem (developed, for instance, in  the recent monograph \cite[Part IV]{Schm17}, see also \cite{MR1303090,MR4379909}) it turns out that the matrix-valued case leads to a number of unexpected new difficulties and problems.

There is an extensive literature about the truncated matrix-valued moment problem for polynomials in a single variable; we mention only a few references such as \cite{Krein,Ando,Dym,
AdamyanT,Zag,FritzscheKM}.
For polynomials in several variables it was treated in \cite{vasilescu,MR2168867,BakonyiW,KimseyW,Kimsey,Le,KimseyT}.
It seems that the truncated matricial moment problem for general measurable or continuous functions has not been yet studied in the literature, according to our knowledge.
Nevertheless  multivariate polynomials are our main guiding example.

Let us briefly describe the contents and main results of this paper.

In Section~\ref{basicconstructions}, we develop some basic constructions on linear mappings $L\colon E\to \cH_q$ and functionals $\Lambda$ on $\cE$. 
In particular, we investigate  the canonical extension  of a linear mapping $L\colon E\to \cH_q$  to a linear functional $\Lambda_L$ on the space $\cH_q(E)$ of Hermitian $q\times q$-matrices over the complexification of $E$.

In Section~\ref{momentfunctionals}, we introduce moment functionals and matrix moment functionals.
Amongst others we obtain a one-to-one correspondence between matrix moment functionals on $E$ and moment functionals on 
$\cH_q(E)$ (Theorem~\ref{LLmomentf} and  Proposition~\ref{equivalence}), so it suffices to study moment functionals on $\cH_q(E)$ or on a general vector space $\cE$. 

In Section~\ref{RTtheorem}, we prove a matrix version of the famous Richter--Tchakaloff theorem (Theorem \ref{richter}).
It states that each moment functional $\Lambda$ on $\cE$ admits a finitely atomic representing measure $\mu\in \cM_q(\cX,\fX)$, so $\Lambda$ can be always represented in the form \eqref{mf1}. D. Kimsey \cite{Kimsey} has obtained such a result in the polynomial case; we treat the general case and use  another approach.
The Richter-Tchakaloff theorem allows one to define and to estimate Carath\'eodory numbers of moment functionals and matrix moment functionals. 

A linear functional $\Lambda$ on $\cE$ is called \emph{strictly positive} if $\Lambda(F)>0$ for all $F\in \cE$, $F\neq O$, such that $F(x)\succeq O$ on $\cX$.
In Section~\ref{strictlypositive}, we show that each strictly positive linear functional is a moment functional (Theorem~\ref{T130}).

In the remaining three sections of the paper we define and investigate  fundamental notions on a moment functional $\Lambda\neq 0$ on $\cE$.

Section~\ref{setofatoms} deals with the set of all atoms $\cW(\Lambda)$ of representing measures of $\Lambda$.
For $x\in \cX$, the vector space $ \bW(\Lambda;x)$ generated by the ranges $\ran  \mu(\{x\})$ for all representing measures $\mu$ of $\Lambda$ is another important quantity. 

A basic concept of the theory of scalar truncated moment problems is the core variety; it was introduced by L.~Fialkow \cite{MR3688460}.
Its importance stems from a result proved in \cite{DioS1} which states that it is equal to the set of atoms.

Apparently, there is no obvious direct generalization of the definition of the core variety to the truncated matricial moment problem.
In Section~\ref{coreset}, we propose the following approach (see Definition~\ref{D147}).
We set $\cV_0(\Lambda):= \cX$ and $\bV_0(\Lambda;x):=\dC^q$ for $x\in\cX$ and define inductively for $j\in \dN$:
\begin{align*}
    \cN_j(\Lambda) &:= \{F\in\ker\Lambda\colon(P_{j-1}FP_{j-1})(x)\succeq  O~~ \text{ for }~ x\in\cV_{j-1}(\Lambda)\},\\ \cV_j(\Lambda)&:=\bigcap\nolimits_{F\in\cN_j(\Lambda)}\{x\in\cV_{j-1}(\Lambda)\colon\det(P_{j-1}FP_{j-1}+P_{j-1}^\bot)(x)=0\},\\   
    \bV_j(\Lambda;x)&:=\bigcap\nolimits_{F\in\cN_j(\Lambda)}\, \ker\, (P_{j-1}FP_{j-1}+P_{j-1}^\bot)(x)~~\text{for  }~ x\in\cV_{j-1}(\Lambda),
\end{align*}
where $P_j\colon\cV_j(\Lambda)\to\cH_q$ denotes the orthogonal projection onto ${\bV_j(\Lambda;x)}$.
Then the \emph{core set} of $\Lambda$ is defined by
\[
    \cV( \Lambda)
    :=\bigcap\nolimits_{j=0}^\infty \cV_j(\Lambda).\label{cdef}
\]
Further, we set\, $\bV( \Lambda;x):=\bigcap_{j=0}^\infty \bV_j(\Lambda;x)$ for  $x\in\cV( \Lambda)$. 

In Section~\ref{coreset}, we prove two main results of this paper (Theorems~\ref{P148-c} and~\ref{T149}). The proofs are lengthy and require a number of  technical  considerations. The first of these theorems states that 
the sets $\cV_j(\Lambda)$ and $\cV(\Lambda)$ are measurable and 
there exists a $k\in\dN$ such that
\[
    \cV_k(\Lambda)
    =\cV_{k+\ell}(\Lambda)
    =\cV( \Lambda)
    \text{ and }
    \bV_{k}(\Lambda;x)
    =\bV_{k+\ell}(\Lambda;x)
    =\bV( \Lambda;x)
    \neq \{0\}
\]
for $x\in\cV( \Lambda)$ and all $\ell\in\dN$.
The second says that 
\[
\cW(\Lambda)
=\cV( \Lambda)
\text{ and }
\bW(\Lambda;x)
=\bV( \Lambda;x)
\text{ for }x\in\cW(\Lambda).
\]

In Section~\ref{someexamples}, we illustrate these notions by developing three simple examples.

In Section~\ref{prelim}, we collect some notation and basic facts which are used later.

The following assumptions and notations are kept throughout this paper:

\begin{itemize}
    \item $q$ is a positive integer and $\cH_q$ are the complex Hermitian $q\times q$-matrices,
    \item $(\cX,\fX)$ is a measurable space such that $\{x\}\in \fX$ for all $x\in \cX$,
    \item $E$ is a \textbf{finite-dimensional} real vector space of  measurable functions ${f\colon\cX\to \dR}$,
    \item $\cE$ is a \textbf{finite-dimensional} real vector space of measurable mappings ${F\colon \cX\to \cH_q}$.
\end{itemize}  

\section{Preliminaries}\label{prelim}
We abbreviate $\dN_0=\{0,1,2,\dotsc\}$ and $\dN=\{1,2,\dotsc\}$.
For $p\in \dN$, we denote by $M_p(\dC)$ the $p\times p$-matrices with complex entries, by $\cH_p$ the Hermitian matrices of $M_p(\dC)$ and by $\cH_{p,\succeq}$ the nonnegative matrices of $\cH_p$.
The Loewner ordering of Hermitian matrices is denoted by ``$\succeq$''. 

As noted above, we let $q\in \dN$.
Let $e_{jk}$, where $ j,k\in\{1,\dotsc,q\}$, denote the standard matrix units of $M_q(\dC)$, that is, $e_{jk}$ is the matrix with $1$ as $(j,k)$-entry and zeros elsewhere.
For $j,k\in\{1,\dotsc,q\}$ we set
\[
    H_{jk}
    :=
    \begin{cases}
        \frac{1}{\sqrt{2}}(e_{jk}+e_{kj})&\text{if }j<k\\
        ~ e_{jj}&\text{if }j=k\\
        \frac{\ii}{\sqrt{2}}(e_{jk}-e_{kj})&\text{if }j>k
    \end{cases}.
\]
Let $\tr $ denote the trace  and $ \langle \cdot,\cdot\rangle$  the scalar product on $M_q(\dC)$ given by
\[
\langle X,Y\rangle
:= \tr (X^*Y)\quad\text{for }X,Y\in M_q(\dC).
\]
For Hermitian matrices $X,Y\in \cH_q$ the scalar product
\begin{equation}\label{scalarher}
\langle X,Y\rangle
= \tr (XY)
=\tr (YX)
=\langle Y,X\rangle
\end{equation}
is real and $(\cH_q,\langle \cdot,\cdot\rangle)$ is a $q^2$-dimensional real Hilbert space with orthonormal basis $\{H_{jk}\colon j,k=1,\dotsc,q\}$.
This is the main reason why we shall work with the matrices $H_{jk}$ instead of the matrix units $e_{jk}$.

For  matrices $X,Y\in \cH_{q,\succeq}$ the scalar product is nonnegative, since
\begin{equation}\label{scalarpos}
\langle X,Y\rangle
= \tr (XY)
=\tr (\sqrt{Y}\, X \sqrt{Y})
\geq 0.
\end{equation} 
The standard scalar product of $\dC^q$ is also denoted by $\langle\cdot,\cdot\rangle$, that is, 
\[
\langle u,v\rangle:=\sum_{j=0}^q u_j\,\overline{v_j}\quad\text{for  }u=(u_1,\dotsc,u_q)^\trn , v=(v_1,\dotsc,v_q)^\trn \in \dC^q.
\]
For later use we recall the following well-known lemma.

\begin{lem}\label{tracezero}
Suppose that $A, B\in \cH_{q,\succeq}$.
Then the following are equivalent:
\begin{itemize}
\item[(i)] $\langle A,B\rangle=0$.
\item[(ii)] $ABA=O$.
\item[(iii)] $BA=O$.
\item[(iv)] $AB=O$, or equivalently,  $A \lceil \, \ran B=O$.
\item[(v)] $\ran A \perp \ran B$.
\item[(vi)]  $A$ acting  on $\dC^q=\ran B\,\oplus\, \ker B$ has a block matrix representation 
\begin{align*}
A=
\begin{pmatrix}
O&O \\
O& \tilde{A}
\end{pmatrix}.
\end{align*}
\end{itemize}
\end{lem}
\begin{proof}
(i)$\to$(ii):
Since   $\langle A,B\rangle=\tr(AB)=\tr \ (\sqrt{A}B\sqrt{A})=0$ by~(i) and $\sqrt{A}\, B\, \sqrt{A}\succeq O$, we obtain $\sqrt{A}\, B\, \sqrt{A}=O$.
Multiplying by $\sqrt{A}$ from the right and from the left yields $ABA=O$.

(ii)$\to$(iii):
Let $u\in \dC^q$.
Then, by~(ii), 
\[
\lVert\sqrt{B}\, Au\rVert^2
=\langle \sqrt{B}\, Au,\sqrt{B}\, Au\rangle
= \langle ABAu,u\rangle
=0.
\]
Hence $\sqrt{B}\lceil  A\dC^q=O$.
Since the closure of $A\dC^q$ and $\sqrt{A}\, \dC^q$ coincide, it follows that $\sqrt{B}\sqrt{A}=O$.
Multiplying by $\sqrt{B}$ and $\sqrt{A}$ from the left and the right, respectively, we get $BA=O$.

(iii) and~(iv) are equivalent by applying the adjoint and~(iv)$\to$(i) is trivial.

(iv)$\leftrightarrow$(v) follows from  the identity $\langle Au,Bv\rangle=\langle u ,ABv\rangle$, $u,v\in \dC^q$.

(vi)$\to$(iv) is obvious.
Using that $A\succeq O$ it follows that~(iv)$\to$(vi).
\end{proof}

Let $\scrL(E,\cH_q)$ denote the linear mappings of $E$ into $\cH_q$.
We define a real vector space $\cH_q(E)$  and a cone $\cH_{q,\succeq}(E)$ in $\cH_q(E)$ by
\begin{align*}
    \cH_q(E)&:=\left\{\sum\nolimits_{j,k=1}^p\, f_{jk}H_{jk}\colon f_{jk}\in E\text{ for }j,k=1,\dotsc,q\right\},\\
    \cH_{q,\succeq}(E)&:=\{ F\in \cH_q(E)\colon F(x)\succeq O\text{ for all }x\in \cX\}.
\end{align*}
That is, $\cH_q(E)$ are the $q\times q$-Hermitian matrices over the complexification of the real vector space $E$.
Clearly, if $\{f_1,\dotsc,f_m\}$ is a  basis of the vector space $E$, then $\{f_i H_{jk}\colon i=1,\dotsc,m; j,k=1,\dotsc,q\}$ is a vector space basis of $\cH_q(E)$.
In particular, $\dim \cH_q(E)= q^2\dim E$.

Next we develop a number of measure-theoretic definitions and facts.
Suppose $(\cX,\fX)$ is a  non-trivial measurable space such that $\{x\}\in\fX$ for all $x\in\cX$. 
For  a subset $\cY$ of $\cX$, the system $\cY\cap\fX:=\{\cY\cap X\colon X\in\fX\}$ is a $\sigma$-algebra.
If $\cY\in\fX$, then $\cY\cap\fX$ is even a sub-$\sigma$-algebra of $\fX$.

We denote by $\cM_q(\cX;\fX)$ the $\cH_q$-valued positive  measures on $(\cX,\fX)$.
Suppose that $\mu\in \cM_q(\cX,\fX)$.
Let $\tau:=\tr \mu$ be the trace measure of $\mu$.
Then there exists a Radon--Nikodym matrix $\rnd(x)=(\rnde_{ij}(x))_{i,j=1}^q$ of measurable functions $\rnde_{ij}$ on $(\cX,\fX)$ such that the matrix $\rnd(x)$ is Hermitian, that is, $\rnde_{jk}(x)=\ov{\rnde_{kj}(x)}$\, for $j,k=1,\dotsc,q$, and $ \rnd(x)\succeq O$  $\tau$-a.\,e.\ on $\cX$.
Upon changing $\rnd$ on a  $\tau$-null set if necessary, we will assume  the following throughout this paper:
Whenever we speak about the Radon--Nikodym matrix of $\mu$ we mean the version $\rnd(x)$ for which $\rnd(x)$ is Hermitian and $\rnd(x)\succeq O$ for all $x\in \cX$.
Further, we will always use the notation $\tau$ for the tracial measure of $\mu$ and $\rnd$ for the Radon-Nikodym matrix  with respect to $\tau$.

A point $x\in \cX$ is called an \emph{atom} of $\mu\in \cM_q(\cX,\fX)$ if $\mu(\{x\})\neq  O$.
A measure $\mu\in \cM_q(\cX,\fX)$ is said to be \emph{finitely atomic} if there exists a finite subset $M$ of $\cX$ such that $\mu(\cX\setminus M)=O$.  

Let $x\in \cX$.
Then $\delta_x$ means the point measure at $x$, that is, $\delta_x(M)=1$ if $x\in M$ and $\delta_x(M)=0$ if $x\notin M$, and $\ell_x$ denotes the point evaluation at $x$, that is, $\ell_x(F)=F(x)$ for $F\in \cE$.

Let $\cM(\cX,\fX;\cH_q)$ be the real vector space of measurable mappings $F\colon\cX\to\cH_q$.
A measurable mapping $F\colon \cX\to\cH_q$ on $(\cX,\fX)$ is said to be \emph{$\mu$-integrable} if $\langle F,\rnd\rangle\in L^1(\tau;\dR)$.
In this case, $\int_{\cX} \langle F,\dif\mu\rangle$ is defined by
\begin{equation}\label{inffmu}
    \int_{\cX} \langle F,\dif\mu\rangle
    :=\int_{\cX }\langle F,\rnd\rangle\, \dif\tau.
\end{equation}
Let $L^1(\mu;\cH_q)$ denote the set of all $\mu$-integrable mappings $F\colon \cX\to \cH_q$.

\begin{exm}\label{E1700}
Let $k\in\dN$, $x_1,\dotsc,x_k\in\cX$, and $M_1,\dotsc,M_k\in\cH_{q,\succeq}$.
Suppose that $M_j\neq O$ for $j=1,\dotsc,k$. 
Then $\mu:=\sum_{j=1}^k M_j\delta_{x_j}$ belongs to $\cM_q(\cX,\fX)$.
Clearly, the trace measure of $\mu$ is $\tau:=\sum_{j=1}^k \tr (M_j)\delta_{x_j}$ and the Radon--Nikodym derivative of $\mu$ with respect to $\tau$ is
\[
    \rnd(x)
    =\sum_{j=1}^k \, (\tr \,M_j )^{-1} \, M_j\mathbf{1}_{\{x_j\}}(x),
\]
where $\mathbf{1}_{\{x_j\}}$ denotes the characteristic function of the singleton $\{x_j\}$. 
Then, by \eqref{inffmu} and \eqref{scalarher} we obtain for any $F\in L^1(\mu;\cH_q)$,
\begin{equation}\label{inttrace}
    \int_\cX \langle F,\dif\mu\rangle
    =\int_\cX \langle F,\rnd\rangle\dif\tau
    =\sum_{j=1}^k\, \langle F(x_j),M_j\rangle=\sum_{j=1}^k\, \tr (F(x_j)M_j).
\end{equation}
\end{exm}

\section{Some Basic Constructions}\label{basicconstructions}
To each linear mapping $L\in \scrL(E,\cH_q)$ we  associate in this section a linear functional $\Lambda_L$ and a linear mapping $L_{\otimes}$.

\begin{dfn}
For $L\in \scrL(E,\cH_q)$, we define a linear functional $\Lambda_L$ on $\cH_q(E)$  by 
\begin{equation}\label{defcl}
    \Lambda_L(F)
    :=\sum_{j,k=1}^q\langle L( \langle F,H_{jk}\rangle), H_{jk}\rangle,
    \qquad F\in \cH_q(E)
\end{equation}
and a linear mapping $L_{\otimes}\colon\cH_q(E)\to \cH_q \otimes \cH_q$  by
\[
    L_{\otimes}(F)
    :=\sum_{j,k=1}^q H_{jk} \otimes  L(\langle F,H_{jk}\rangle), 
    \qquad F\in\cH_q(E).
\]
\end{dfn}

\begin{rem}
If $\mathrm{m}\colon M_q(\dC)\otimes M_q(\dC)\to M_q(\dC)$ denotes the multiplication map, then
\[
    \Lambda_L(F)
    =\tr[(\mathrm{m}\circ L_\otimes)(F)],
    \qquad F\in\cH_q(E).
\]
\end{rem} 

Let $I_q$ be the identity matrix in $M_q(\dC)$.
The following lemmas collect a number of useful formulas.

\begin{lem}\label{propertyL}
Let $x\in \cX$, $f\in E$, $F\in \cH_q(E)$, $H\in \cH_q$, and $v\in \dC^q$.
Then:
\begin{enumerate}[(a)]
    \item\label{propertyL.i} $fH\in \cH_q(E)$ and $\Lambda_L(fH)=\langle L(f),H\rangle$.
    In particular, $\Lambda_L(f I_q)=\tr L(f)$.
    \item\label{propertyL.ii} $H\ell_x\in \scrL(E,\cH_q)$ and $\Lambda_{H\ell_x}(F)=\langle F(x),H\rangle$.
    In particular, $\Lambda_{I_q\ell_x}(F)=\tr F(x)$.
    \item\label{propertyL.iii} $vv^*\ell_x\in \scrL(E,\cH_q)$ and $\Lambda_{vv^*\ell_x}(F)=v^* F(x)v$.
    \item\label{propertyL.iv} $fvv^*\in \cH_q(E)$ and $\Lambda_L(fvv^*)= v^*L(f)v$.
\end{enumerate}
\end{lem}
\begin{proof} 
We can  write $H=\sum_{j,k=1}^qh_{jk}H_{jk}$ and $F=\sum_{j,k=1}^qf_{jk}H_{jk}$, with  $h_{jk}\in\dR$ and $f_{jk}\in E$.
Then  $h_{jk}=\langle H,H_{jk}\rangle$ and $f_{jk}=\langle F,H_{jk}\rangle$ for all $j,k$.

\ref{propertyL.i}:
It is obvious that $fH\in \cH_q(E)$.
We  compute
\[\begin{split}
    \Lambda_L(fH)
    &=\sum_{j,k=1}^q \langle L( \langle fH, H_{jk}\rangle),H_{jk}\rangle
    =\sum_{j,k=1}^q\langle L( f\langle H, H_{jk}\rangle), H_{jk}\rangle\\
    &=\sum_{j,k=1}^q \langle L(fh_{jk}),H_{jk}\rangle
    =\sum_{j,k=1}^q\langle h_{jk}L(f),H_{jk}\rangle\\
    &=\left\langle L(f),\sum_{j,k=1}^qh_{jk} H_{jk}\right\rangle
    =\langle L(f),H\rangle.
\end{split}\]
Clearly, $\Lambda_L(f I_q)= \langle L(f),I_q\rangle= \tr L(f)$.

\ref{propertyL.ii}: 
Recall that the matrices $\{H_{jk}\}$ are an orthonormal basis of $(\cH_q,\langle\cdot, \cdot\rangle)$.
Therefore, we have $F(x)=\sum_{j,k=1}^qf_{jk}(x)H_{jk}$ with real-valued functions $f_{jk}(x)$ on $\cX$.
Using \eqref{scalarher} we obtain
\begin{equation}\label{HF(x)}
    \langle H, F(x)\rangle
    = \langle F(x),H\rangle
    = \sum_{j,k=1}^q \ov{f_{jk}(x)}\, h_{jk}.
\end{equation}
On the other hand,
\[\begin{split}
\Lambda_{H\ell_x}(F)
&= \sum_{j,k=1}^q \langle  H\ell_x(\langle F,H_{jk}\rangle) ,H_{jk}\rangle
= \sum_{j,k=1}^q \langle (H\ell_x)(f_{jk}),H_{jk}\rangle\\ 
&= \sum_{j,k=1}^q \langle Hf_{jk}(x),H_{jk}\rangle
= \sum_{j,k=1}^q  \ov{f_{jk}(x)}\langle H,H_{jk}\rangle
=\sum_{j,k=1}^q \ov{f_{jk}(x)}\, h_{jk}.
\end{split}\]
Inserting \eqref{HF(x)} into the latter equation yields $\Lambda_{H\ell_x}(F)=\langle H, F(x)\rangle$.

Setting $H=I_q$ we get $\Lambda_{I_q\ell_x}(F)=\langle I_q,F(x)\rangle=\tr F(x)$. %

\ref{propertyL.iii}:
Clearly, 
$vv^*\ell_x\in \scrL(E,\cH_q)$.
Using~\ref{propertyL.ii} we derive
\[
    \Lambda_{ vv^*\ell_x}(F)
    =\langle  vv^*, F(x)\rangle
    =\tr(vv^* F(x))
    = \tr(v^*F(x)v)
    = v^*F(x)v.
\]

\ref{propertyL.iv}:
Obviously, $fvv^*\in \cH_q(E)$.
Using~\ref{propertyL.i} we deduce 
\[
\Lambda_L(fvv^*)
=\langle L(f),vv^*\rangle
=\tr (L(f)vv^*)
= \tr (v^*L(f)v)
= v^*L(f)v.\qedhere
\]
\end{proof}

\begin{lem}\label{preliminaryL}
The map $L\mapsto \Lambda_L$ is an isomorphism of the vector space $\scrL(E,\cH_q)$ on the dual space $\cH_q(E)^*$ of the vector space $\cH_q(E)$.
\end{lem}
\begin{proof}
Clearly, $L\mapsto \Lambda_L$ is a linear map of $\scrL(E,\cH_q)$ in the vector space $\cH_q(E)^*$.
Since $\dim \cH_q(E)=q^2\dim E$, we have $\dim \cH_q(E)^*=q^2 \dim E=\dim \scrL(E,\cH_q)$.
It suffices to show that this map is injective.
Indeed, suppose that $\Lambda_L=0$ for some $L\in \scrL(E,\cH_q)$.
Let $f\in E$ and set $F:=fH_{rs}$, where  $r,s\in \{1,\dotsc,q\}$.
Then, by Lemma~\ref{propertyL}\ref{propertyL.i},
$0=\Lambda_L(fH_{rs})=\langle L(f),H_{rs}\rangle$.
Thus $L=O$, because the matrices $\{H_{rs}\}$ form an orthonormal basis of $(\cH_q,\langle \cdot,\cdot\rangle)$. This proves that the map $L\mapsto \Lambda_L$ is injective.
\end{proof}

\begin{lem}\label{llotimes}
Let $L\in \scrL(E,\cH_q)$ be a linear map 
and $e:=\sum_{j=1}^q e_j\otimes e_j$, where $e_1,\dotsc,e_q$ is the standard basis of $\dR^q$.
Then, for any $F\in \cH_q(E)$,
\begin{align}
L_{\otimes}(F^\trn )&=(L_{\otimes}^\trn (F))^\trn ,\label{LT}\\
\Lambda_L(F)&=e^*L_{\otimes}^\trn (F)e.\label{LTe}
\end{align}
\end{lem}
\begin{proof} First we note that we always identify $1\times 1$-matrices such as $e^*L_{\otimes}^\trn (F)e$, $e_r^* H_{jk}e_s$ etc.\ with the corresponding numbers. 
Let $F\in \cH_q(E)$.
As usual, we write $F=\sum_{j,k=1}^qf_{jk} H_{jk}$, with $f_{jk}\in E$.
For $j\in \dZ$ we set $\epsilon (j):=1$ if $j\geq 0$ and $\epsilon(j):=-1$ if $j<0$.
Then, for $j,k=1,\dotsc,q$, we have $H_{jk}^\trn =\epsilon(k{-}j) H_{jk}$ by the definition of $H_{jk}$ and hence
\[
    \langle F^\trn ,H_{jk}\rangle
    =\tr( F^\trn H_{jk})
    = \tr (F^\trn H_{jk})^\trn 
    =\epsilon(k{-}j)\, \tr( H_{jk}F)
    =\epsilon(k{-}j)\langle F,H_{jk}\rangle,
\]
so that
\[\begin{split}
    H_{jk}\otimes  L(\langle F^\trn ,H_{jk}\rangle)
    &= \epsilon(k{-}j) H_{jk}\otimes L(\langle F,H_{jk}\rangle)\\
    &= H_{jk}^\trn \otimes L(\langle F,H_{jk}\rangle)
    =[H_{jk}\otimes L^\trn (\langle F,H_{jk}\rangle)]^\trn .
\end{split}\]
This in turn implies that\, $L_{\otimes} (F^\trn )= (L_{\otimes}^\trn (F))^\trn $\,  which proves \eqref{LT}.

Now we verify \eqref{LTe}.
We abbreviate $L_{jk}:=L(\langle F,H_{jk}\rangle )=L(f_{jk})$.
Then, by the corresponding definitions,  $L_{\otimes}^\trn  (F)=\sum_{j,k=1}^q H_{jk}\otimes L_{jk}^\trn $ and $\Lambda_L(F)=\sum_{j,k=1}^q\langle L_{jk},H_{jk}\rangle$.
Therefore,
\[\begin{split}
    e^*L_{\otimes}^\trn (F)e
    &=\biggl( \sum_{r=1}^q e_r^*\otimes e_r^* \biggr)   \biggl(\sum_{j,k=1}^q H_{jk}\otimes L_{jk}^\trn \biggr) \biggl( \sum_{s=1}^q e_s\otimes e_s\biggr)\\
    & = \sum_{j,k=1}^q \sum_{r,s=1}^q (e_r^* H_{jk}e_s\otimes e_r^*L_{jk}^\trn e_s)
    =\sum_{j,k=1}^q \sum_{r,s=1}^q (e_r^* H_{jk}e_s)\cdot (e_r^*L_{jk}^\trn e_s).
\end{split}\]
Let us denote the $(n,m)$ entry of a matrix $A$ by $A_{nm}$.
Then we have $e_r^*H_{jk}e_s= (H_{jk})_{rs}$ and $e_r^*L_{jk}^\trn e_s= (L_{jk}^\trn )_{rs}=(L_{jk})_{sr}$, so that
\[\begin{split}
    e^*L_{\otimes}^\trn (F)e 
    &=\sum_{j,k=1}^q\sum_{r,s=1}^q (H_{jk})_{rs} (L_{jk})_{sr}
    =\sum_{j,k=1}^q \tr (L_{jk}H_{jk})\\
    &=\sum_{j,k=1}^q \langle L_{jk},H_{jk}\rangle
    =\Lambda_L(F).\qedhere
\end{split}\]
\end{proof}

\begin{cor}\label{LotimesimpliesL}
Suppose  $L\in \scrL(E,\cH_q)$ is a linear map such that $L_{\otimes} (F)\succeq O$ for all $F\in \cH_{q,\succeq}(E)$.
Then $\Lambda_L(F)\geq 0$ for all $F\in \cH_{q,\succeq}(E)$.
\end{cor}
\begin{proof}
Take $F\in \cH_{q\succeq}(E)$.
Since  $F^\trn $ is also in $ \cH_{q,\succeq}(E)$, we have $L_{\otimes}(F^\trn )\succeq O$ and hence $(L_{\otimes}(F^\trn ))^\trn \succeq O$.
Let $e$ be as in Lemma~\ref{llotimes}.
Using \eqref{LTe} and \eqref{LT} we obtain
\[
\Lambda_L(F)
=e^*L_{\otimes}^\trn (F)e
=e^*(L_{\otimes}(F^\trn ))^\trn e
\geq 0.  \qedhere
\]
\end{proof}

\begin{rem}
Formula \eqref{LTe} is only a special case of the general formula \eqref{genLTe} which we state here without proof.
Let $x_1,\dotsc,x_q\in \dC^q$, written as columns.
We set $X:=\sum_{r=1}^q x_re_r^\trn \in M_q(\dC)$ and $x:=Q\, \vvec (X)$, where $\vvec (X):=\sum_{r=1}^q e_r\otimes x_r\in \dR^q\otimes \dR^q$ and $Q:=\sum_{j=1}^q e_j^\trn \otimes I_q\otimes e_j$.
Then 
\begin{equation}\label{genLTe}
\Lambda_L(X^*F X)
= x^*(L_{\otimes}^\trn (F))x
\quad\text{for }F\in \cH_q(E).
\end{equation}
Note that the matrix $Q$ is a permutation matrix which exchanges the order of the Kronecker product, that is, $B\otimes A=Q^\trn (A\otimes B)Q$ for $A,B\in M_q(\dC)$.
\end{rem}

\section{Matrix moment functionals and moment functionals}\label{momentfunctionals}

\begin{dfn}\label{D1012}%
Let $\mu\in \cM_q(\cX,\fX)$ be such  that $E\subseteq L^1(\tr\mu;\dR)$.
Then  the linear map $L^\mu \colon E\to \cH_q$ defined by
\begin{equation}\label{matrixmf}
    L^\mu(f)
    :=\int_\cX f\, \dif\mu,
    \quad f\in E,
\end{equation}
is called a \emph{matrix-valued moment functional} on $E$,  briefly a \emph{matrix moment functional}.
\end{dfn}

If $L$ is a matrix moment functional, then a measure $\mu\in \cM_q(\cX,\fX)$  such  that $L=L^\mu$ is called a \emph{representing measure} of $L$.
The corresponding set is
\[
    \cM_L
    := \{\mu\in\cM_q(\cX,\fX)\colon E\subseteq L^1(\tr\mu;\dR)~ \text{ and }~ L=L^\mu\, \}.
\]
and the set of \emph{finitely atomic} representing measures of $L$ is denoted by $\cM^\mathrm{fa}_L$. 
The set of all matrix moment functionals on $E$ is
\[
    \cL_q(E)
    :=\{L^\mu\colon\mu\in\cM_q(\cX,\fX)~ \text{ such that }~ E\subseteq L^1(\tr\mu;\dR)\}.
\]
For a matrix moment functional  the following theorem describes $L_{\otimes}$ and $\Lambda_L$  in terms of a representing measure.  

\begin{thm}\label{LLmomentf}
Suppose that $L\in \scrL(E,\cH_q)$ is a matrix moment functional and $\mu\in \cM_L$.
Let $\tau$ be the trace measure of $\mu$ and $\rnd$ the Radon--Nikodym derivative of $\mu$ with respect to $\tau$.
Then, for $F\in \cH_q(E)$,
\begin{align}
    L_{\otimes}(F)&=\int_\cX (F\otimes \rnd)\, \dif\tau,\label{GFformula}\\ 
    \Lambda_L(F)&= \int_\cX \langle F,\rnd\rangle\, \dif \tau=\int_\cX \tr (F\rnd)\, \dif \tau.\label{GFformula1}
\end{align}
Further,
\begin{equation}\label{poscond}
    L_{\otimes} (F)\succeq O
    \quad\text{and}  
    \quad\Lambda_L(F)\geq 0
    \quad\text{for }F\in \cH_{q,\succeq}(E).
\end{equation}
\end{thm}
\begin{proof} 
Let $F\in\cH_q(E)$.
We can write $F$ as $F=\sum_{j,k=1}^qf_{jk} H_{jk}$ with $f_{jk}\in E$.
Then, for  $j,k=1,\dotsc,q$, we have
\begin{equation}\label{auxequ}
    L(f_{jk})
    =L^\mu(f_{jk})
    =\int_\cX f_{jk}\, \dif\mu
    =\int_\cX f_{jk}\rnd\, \dif\tau
\end{equation}
and therefore
\[
    H_{jk}\otimes L(f_{jk})
    =H_{jk}\otimes\biggl(\int_\cX f_{jk}\rnd\,\dif\tau\biggr)
    =\int_\cX [(f_{jk}H_{jk})\otimes \rnd]\, \dif\tau.
\]
Using the equation $f_{jk}=\langle F,H_{jk}\rangle $ we compute
\[\begin{split}
    L_\otimes (F)
    &=\sum_{j,k=1}^q H_{jk}\otimes L(\langle F, H_{jk} \rangle)        
    =\sum_{j,k=1}^q H_{jk}\otimes L(f_{jk})\\
    &=\sum_{j,k=1}^q\int_\cX [(f_{jk} H_{jk})\otimes \rnd]\,\dif\tau
    =\int_\cX\sum_{j,k=1}^q [(f_{jk} H_{jk})\otimes \rnd]\,\dif\tau\\
    &=\int_\cX\biggl[\biggl(\sum_{j,k=1}^q f_{jk}H_{jk}\biggr)\otimes \rnd\biggr]\, \dif\tau
    =\int_\cX (F\otimes \rnd)\, \dif\tau,
\end{split}\]
which proves \eqref{GFformula}.

Next, using \eqref{auxequ}, we derive
\[\begin{split}
    \langle L(f_{jk}), H_{jk}\rangle
    &=\tr( L (f_{jk})H_{jk})
    =\tr( H_{jk} L (f_{jk}))\\
    &=\tr \biggl( H_{jk}\int_\cX f_{jk}\rnd\, \dif\tau\biggr)
    =\tr \biggl(\int_\cX f_{jk} H_{jk}\rnd\, \dif\tau\biggr)
\end{split}\]
and hence
\[\begin{split}
    \Lambda_L(F)
    &=\sum_{j,k=1}^q \langle L(\langle F,H_{jk})\rangle, H_{jk}\rangle
    =\sum_{j,k=1}^q\langle L (f_{jk}),H_{jk}\rangle\\
    &=\sum_{j,k=1}^q \tr\biggl(\int_\cX f_{jk} H_{jk}\rnd\, \dif\tau\biggr)
    =\int_\cX\tr\biggl(\biggl(\sum_{j,k=1}^q f_{jk} H_{jk}\biggr)\rnd\biggr)\, \dif\tau\\
    &=\int_\cX \tr(F\rnd)\, \dif\tau
    =\int_\cX \langle F, \rnd\rangle \, \dif\tau.
\end{split}\]
This proves \eqref{GFformula1}.

Finally, we verify \eqref{poscond}.
Suppose $F\in \cH_{q,\succeq}(E)$.
Then, by definition, $F(x)\in\cH_{q,\succeq}$ for all $x\in\cX$.
Since  $\rnd(x)\in \cH_{q,\succeq}$, we have $F(x)\otimes \rnd(x)\succeq O$ and hence $L_{\otimes} (F)=\int_\cX (F\otimes \rnd)\,\dif\tau \succeq O$\,  by \eqref{GFformula}. 

From $F(x)\succeq O$ and $\rnd(x)\succeq O$, we obtain $\langle F(x),\rnd(x)\rangle \geq 0$ for $x\in \cX$ by \eqref{scalarpos}. 
Therefore, $ \Lambda(F)=\int_\cX \langle F,\rnd\rangle\, \dif\tau\geq 0$ by \eqref{GFformula1}.
\end{proof}

Now we turn to moment functionals on the vector space $\cE$.

\begin{dfn}\label{defmf}
Let $\mu\in \cM_q(\cX,\fX) $ and suppose that $\cE\subseteq L^1(\mu;\cH_q)$.
Then the linear functional $\Lambda^\mu$ on $\cE$ defined by
\begin{equation}\label{deflammbdamu}
    \Lambda^\mu(F)
    :=\int_\cX \langle F,\dif\mu\rangle
    =\int_\cX \langle F,\rnd\rangle\dif\tau, 
\qquad F\in \cE,
\end{equation}
is called a \emph{moment functional} on $\cE$.
Here $\tau $ denotes the trace measure of $\mu$ and $\rnd$ is the Radon--Nikodym derivative of $\mu$ with respect to $\tau$.
\end{dfn}

We point out that an analogous terminology can be found in \cite{Le} under the name \emph{tracial $K$\nobreakdash-moment functional}, where reference is made to \cite{MR3011271}.

For a moment functional $\Lambda$ on $\cE$, the set of all \emph{representing measures} is
\[
    \cM_\Lambda
    :=\{\mu\in\cM_q(\cX,\fX)\colon\cE\subseteq L^1(\mu;\cH_q)~\text{ and }~ \Lambda=\Lambda^{\mu}\}
\]
and the set of \emph{finitely atomic} representing measures of $\Lambda$ is denoted by $\cM^\mathrm{fa}_\Lambda$. 
The set of all moment functionals on $\cE$ is
\[
    \fL(\cE)
    :=\{\Lambda^\mu\colon\mu\in\cM_q(\cX,\fX)~ \text{ such that }~ \cE\subseteq L^1(\mu;\cH_q)\}.
\]
Further, we let
\[
    \cE_\succeq
    :=\{ F\in \cE\colon F(x)\succeq O~ \text{ for all }~ x\in \cX\}.
\]
First we record a simple lemma.

\begin{lem}\label{L1036+}
If $\Lambda \in \fL(\cE)$, then $\Lambda (F)\geq 0$ for all $F\in \cE_\succeq$.
\end{lem}
\begin{proof} 
Since $F(x)\succeq O$ for $F\in \cE_\succeq$ and $\rnd(x)\succeq O$, the assertion follows at once by combining \eqref{deflammbdamu} and \eqref{scalarpos}. 
\end{proof}

\begin{prop}\label{equivalence}
Let $L\in \scrL(E,\cH_q)$.
Then $L$ is a matrix moment functional on $E$ (according to Definition~\ref{matrixmf}) if and only if $\Lambda_L$ is a moment functional on $\cH_q(E)$ (according to Definition~\ref{defmf}).
In this case, $\cM_L=\cM_{\Lambda_L}$.
\end{prop}
\begin{proof}
If $L$ is a matrix moment functional and $\mu\in \cM_L$, then combining formulas \eqref{GFformula1} and \eqref{deflammbdamu} yields $\Lambda_L=\Lambda^\mu$, so that $\Lambda_L\in \fL(\cE)$ and $\mu\in \cM_{\Lambda_L}$. 

Conversely, suppose  $\Lambda_L\in \fL(\cE)$ and $\mu\in \cM_{\Lambda_L}$.
As usual, let $\rnd$ be the Radon--Nikodym derivative  with respect to the trace measure $\tau$ of $\mu$.
Using Lemma~\ref{propertyL}\ref{propertyL.iv} and formulas \eqref{deflammbdamu} 
and \eqref{matrixmf} 
we derive
\[\begin{split}
    v^* L(f)v
    &=\Lambda_L(fvv^*)
    =\Lambda^\mu(fvv^*)
    =\int_\cX \langle fvv^*,\rnd\rangle\, \dif\tau
    =\int_\cX \tr ((fvv^*)\rnd)\, \dif\tau\\
    &=\int_\cX f\, \tr (v^*\rnd v)\, \dif\tau
    =\int_\cX f\, v^*\rnd v\, \dif\tau
    =\int_\cX v^*(f\rnd )v\, \dif\tau\\
    &=v^*\biggl(\int_\cX f\rnd \dif\tau\biggr)v
    =v^*\biggl(\int_\cX f\dif\mu\biggr)v
    =v^*L^\mu(f)v
\end{split}\]
for all $f\in E$ and $v\in\dC^q$.
This implies that $L=L^\mu$ and $\mu\in\cM_L$.
\end{proof}

Proposition~\ref{equivalence} shows that the study of a matrix moment functional $L$ on $E$ is equivalent to the study of the moment functional $\Lambda_L$ on the special space $\cE:=\cH_q(E)$.
In the sequel  we will study moment functionals $\Lambda$ on a general vector space $\cE$.

In Section~\ref{RTtheorem} we prove that each moment functional has a finitely atomic representing measure.
Hence the following example is important.

\begin{exm}[Finitely atomic measure---Example~\ref{E1700} continued]\label{finatomic}
Let $k\in \dN$, $x_1,\dotsc,x_k\in \cX$, and $M_1,\dotsc,M_k\in \cHp$.
Suppose that $M_j\neq 0$ for  $j=1,\dotsc,k$.
Define a measure $\mu\in \cM_q(\cX,\fX)$ by
\[
    \mu
    :=\sum\nolimits_{j=1}^k M_j\delta_{x_j}.
\]
The trace measure $\tau$ and the Radon--Nikodym matrix $\rnd $ have been described in  Example~\ref{E1700}.
By \eqref{inttrace}, the moment functional $\Lambda^\mu$ on $\cE$  is given by 
\begin{equation}\label{lqf}
\Lambda^\mu (F)=\sum_{j=1}^k \tr (F(x_j)  M_j)\quad\text{for } F\in \cE.
\end{equation}
Now we consider the corresponding  matrix moment functional $L:=L^\mu$ on $E$.
Then, $L(f)=\sum_{j=1}^k f(x_j)M_j$ for $f\in E$.
Inserting the expressions for $\tau$ and $\rnd $ into \eqref{GFformula}--\eqref{GFformula1} we obtain for any $F\in \cH_q(E)$,  
\begin{align}\label{tracefor}
L_{\otimes}(F)&=\sum_{j=1}^k F(x_j)\otimes M_j ,\\
\Lambda_L(F)&=\sum_{j=1}^k \tr (F(x_j)  M_j).\label{lqf1}
\end{align}
Note that \eqref{lqf} and \eqref{lqf1} are the same formulas.

It should be emphasized  that  $L_{\otimes}$ and  $\Lambda_L$ depend only on the  linear mapping $L\in \scrL(E,\cH_q)$. 
Thus, if $\tilde{\mu}$ is another (not necessarily 
finitely 
atomic) representing measure for the matrix moment functional $L=L^{\mu}$,  the corresponding mapping  $L_{\otimes}$ and functional $\Lambda_L$ are still given by \eqref{tracefor} and \eqref{lqf1}.

Let us return to the  functional $\Lambda^\mu$ on $\cE$.
From Lemma~\ref{tracezero}, $(i)\leftrightarrow(vi)$, it follows that an element $F\in \cE_\succeq$ satisfies $\Lambda^\mu(F)=0$ if and only if for all $j=1,\dotsc,k$  the operator $F(x_j)$ acting  on $\dC^q=\ran M_j\,\oplus\, \ker M_j$ has a block matrix representation 
\begin{equation}\label{formFxj}
F(x_j)=
\begin{pmatrix}
O&O \\
O&\tilde{F}(x_j)
\end{pmatrix}.
\end{equation}
Note that ${\tilde{F}}(x_j)\succeq O$\, on $ \ker M_j$, because $F\in \cE_\succeq$.
(In the case\, $\ker M_j=\{0\}$, \eqref{formFxj} should be read as $F(x_j)=O$.)
\end{exm}

\section{The matricial Richter--Tchakaloff theorem}\label{RTtheorem}
The scalar Richter--Tchakaloff theorem asserts that each moment functional on a finite-dimensional vector space of functions admits a finitely atomic representing measure.
The following theorem and its corollaries are counter-parts of this fundamental result in the matricial case.

\begin{thm}\label{richter}
Suppose that $(\cX,\fX)$ is a measurable space and  $\cE$ is a finite-dimensional real vector space of measurable mappings $F\colon \cX\to \cH_q$.
Let $\mu\in\cM_q(\cX,\fX)$ be such that $\cE\subseteq L^1(\mu;\cH_q)$.
Recall that $\Lambda^\mu$ denotes the corresponding moment functional, that is, $\Lambda^\mu(F) :=\int_\cX \langle F,\dif\mu\rangle$ for $F\in \cE$. Suppose $\Lambda^\mu \neq 0$.

Then there exists an $N$\nobreakdash-atomic  measure\, $\nu=\sum_{j=1}^NM_j\delta_{x_j}$, where\,
\begin{align}\label{Nestimate}
N\leq \dim \cE,
\end{align} $x_1,\dotsc,x_N\in \cX$, $M_1,\dotsc,M_N\in \cH_{q,\succeq}$,  $\rank M_j=1$ for all $j=1,\dotsc,N$,
such that $\Lambda^\mu=\Lambda^\nu$. That is, there are nonzero vectors $v_1,\dotsc,v_N\in \dC^q$ such that $M_j=v_jv_j^*$ for $j=1,\dotsc,N$ and
\begin{equation}\label{finiteat}
    \int_\cX \langle F,\dif\mu\rangle
    =\int_\cX \langle F,\dif\nu\rangle
    \equiv \sum_{j=1}^N \tr(F(x_j)M_j)\equiv \sum_{j=1}^N v_j^*F(x_j)v_j\quad\text{for all }~~F\in\cE.
\end{equation}

Moreover, if  $\cY\in\fX$ is  such that $\mu(\cX\setminus \cY)=O$, then the points $x_1,\dotsc, x_N$ can be chosen from $\cY$.
\end{thm}
\begin{proof}
Let $\tau $ be the trace measure of $\mu$ and $\rnd $  the Radon--Nikodym derivative of $\mu$ with respect to  $\tau$. 
Clearly,
\[
V
:=\{\langle F,\rnd \rangle\colon F\in \cE\}
\]
is a linear subspace of $L^1(\tau;\dR)$ with finite dimension $\dim V \leq m:=\dim \cE$.
Now we  apply the scalar Richter--Tchakaloff theorem \cite[Thm.~1.24]{Schm17} to the measurable space $(\cX,\fX)$, the scalar measure $\tau\in\cM_1(\cX,\fX)$, and the linear functional $L^\tau$ on the finite-dimensional  space $ V $ of $\tau$-integrable real-valued functions  defined by  
\begin{equation}\label{defltau}
L^\tau(f):=\int_\cX f\dif\tau, ~~~ f\in V.
\end{equation}
According to \cite[Thm.~1.24]{Schm17}, then there exist an integer $n\in\dN$, $n\leq \dim V$, numbers $c_1,\dotsc,c_n\in[0,\infty)$, and points $t_1,\dotsc,t_n\in\cX$ such that
\begin{equation}\label{T124.1}
    L^\tau(f)
    =\sum_{j=1}^n~ c_jf(t_j)~~~
    \text{for all }~~f\in V.
\end{equation}
Now we set
\begin{equation}\label{defnu}
    \tilde{\nu}
    :=\sum_{j=1}^n\, R_j \delta_{t_j},\quad\text{where }~~ R_j:=c_j\rnd (t_j)\text{ for }j=1,\dotsc, n.
\end{equation}
Since $\rnd (x)\succeq O$ for all $x\in \cX$ and $c_j\geq 0$,  $R_j\succeq O $ for $j=1,\dotsc,n$.
Hence $\tilde{\nu}\in\cM_q(\cX,\fX)$.

Consider an arbitrary $F\in \cE$.
Then $f:=\langle F,\rnd \rangle$ belongs to $ V $.
Therefore, using \eqref{T124.1}  and \eqref{defnu} we obtain
\[\begin{split}
    \Lambda^\mu(F)
    &=\int_\mathcal{X}\langle F,\rnd \rangle\dif\tau
    =\int_\cX f\dif\tau
    =L^\tau(f) 
    =\sum_{j=1}^n c_j f(t_j)\\
    &=\sum_{j=1}^n c_j\langle F(t_j),\rnd (t_j)\rangle
    =\sum_{j=1}^n~ \langle F(t_j),R_j\rangle
    =\Lambda^{\tilde{\nu}} (F).
\end{split}\]
This proves that $\Lambda^\mu=\Lambda^{\tilde{\nu}}$.

Each matrix $R_j\in \cH_{q,\succeq}$ has an orthonormal basis of eigenvectors.
Thus, for each $j\in\{1,\dotsc,n\}$, there exist numbers $\lambda_{j,1},\dotsc,\lambda_{j,q}\in[0,\infty)$ and vectors $u_{j,1},\dotsc,u_{j,q}\in\dC^q$ such that $R_j=\sum_{i=1}^q\lambda_{j,i}u_{j,i}u_{j,i}^*$.
Hence, for $F\in \cE$,
\[\begin{split}
    \Lambda^\mu(F)&=\sum_{j=1}^n~ \langle F(t_j),R_j\rangle
    =\sum_{j=1}^n\sum_{i=1}^q\lambda_{j,i}\langle F(t_j),u_{j,i}u_{j,i}^*\rangle\\
    &  =\sum_{j=1}^n\sum_{i=1}^q\lambda_{j,i}\, \tr (F(t_j)u_{j,i}u_{j,i}^*)
    =\sum_{j=1}^n\sum_{i=1}^q\lambda_{j,i}\, \tr (u_{j,i}^*F(t_j)u_{j,i})\\
    &=\sum_{j=1}^n\sum_{i=1}^q\lambda_{j,i}\,  u_{j,i}^*F(t_j)u_{j,i}.
\end{split}\]
This shows that  $\Lambda^\mu$  is a non-negative  combination of linear functionals $\ell_{u,x}$ on $\cE$ which are defined by $\ell_{u,x}(F)=u^*F(x)u$, $F\in \cE$, for $x\in \cX$ and $u\in \dC^q$.
Let $\fC$ denote the convex cone in the dual vector space  $\cE^*$ generated by these functionals $\ell_{u,x}$.
Since $\dim(\cE^*)=m$, Carath\'eodory’s theorem (see e.\,g.\ \cite[Thm.~A.35(i), p.~512]{Schm17}), applied to the cone $\fC$, implies that $\Lambda^\mu\in \fC$ is a non-negative combination of at most $m$ functionals $\ell_{u,x}$.
Hence, since $\Lambda^\mu\neq 0$,  there exist $N\in \dN$, $N\leq m$, numbers $\lambda_1,\dotsc,\lambda_N\in(0,\infty)$, non-zero vectors $u_1,\dotsc,u_N\in\dC^q$, and points $x_1,\dotsc,x_N\in\cX$ such that $\Lambda^\mu=\sum_{j=1}^{N}\lambda_j \ell_{u_j,x_j}$.
We set $v_j:=\lambda_{j}^{1/2}u_{j}$ and $M_j:=v_jv_j^*$ for $j=1,\dotsc, N$ and define $\nu=\sum_{j=1}^N M_j\delta_{x_j}$.
Then $\rank M_j=1$ and for $F\in \cE$ we have 
\[
    \Lambda^\mu(F)
    =\sum_{j=1}^N \ell_{v_j,x_j}(F)
    =\sum_{j=1}^N v_j^*F(x_j)v_j
    =\sum_{j=1}^N \tr(F(x_j)v_jv_j^*)
    =\Lambda^\nu(F),
\]
which proves \eqref{finiteat}.

We show the last assertion. 
Let $\fY:=\{ X\cap\cY: X\in\fX\, \}$.
Then $\fY$ is a sub-$\sigma$\nobreakdash-algebra of $\fX$ on $\cY$ and $(\cY,\fY)$ is a measurable space.
Set $\tilde\mu:=\mu\lceil \fY$ and $\tilde{\cE}:=\{ \tilde{F}:=F \lceil \cY\colon F\in \cE\,  \}$.
Then $\tilde\mu\in\cM_q(\cY,\fY)$.
Since $\mu(\cX\setminus \cY)=O$ by assumption,  $\Lambda^\mu (F)=\Lambda^{\tilde{\mu}}(\tilde{F})$ for $F\in \cE$.
Therefore, applying the first assertion to the functional $\Lambda^{\tilde{\mu}}$ on $\tilde{\cE}$ it follows that we can choose the atoms $x_1,\dotsc,x_N$ from the set $\cY$.
\end{proof}

\begin{rem}
In the first part of the preceding proof we have shown the following:
If the functional $L^\tau$ on  $V=\{\langle F,\rnd \rangle\colon F\in \cE\}$ defined by \eqref{defltau} can be given by a (scalar) measure with $n$ atoms, then $\Lambda$ has a representing measure of $\cM_q(\cX,\fX)$ with $n$ atoms.
Depending on $\mu$ and $\cE$ this leads in many cases to better estimates than given in \eqref{Nestimate}.
\end{rem}

The following corollary is a version of the Richter--Tchakaloff theorem  for matrix moment functionals on $E$.

\begin{cor}\label{richtercor}
Suppose   $\mu\in\cM_q(\cX,\fX)$ and  $E$ is a finite-dimensional real subspace of $L^1(\tr\mu;\dR)$.
Let $L^\mu \colon E\to \cH_q$ be the matrix moment functional on $E$ defined by
\[%
    L^\mu(f)
    :=\int_\cX f(x)\, \dif\mu,
    \quad f\in E.
\]%
Suppose  $L^\mu\neq 0$.
Then there exists an $N$\nobreakdash-atomic measure $\nu=\sum_{j=1}^N\, M_j\delta_{x_j}$, with
\begin{equation}\label{estE}
    N
    \leq(\dim E)\, q^2,
\end{equation}
$x_1,\dotsc,x_N\in \cX$,  $M_1,\dotsc,M_N\in \cH_{q,\succeq}$, $\rank M_j=\dotsb=\rank M_N=1$  such that $L^\mu=L^\nu$,  that is,
\begin{equation}\label{Lrepnu} 
    \int_\cX f(x)\, \dif\mu
    =\int_\cX f(x) \, \dif\nu
    \equiv \sum_{j=1}^NM_jf(x_j)\quad\text{for all } f\in E.
\end{equation}
\end{cor}
\begin{proof}
We apply Theorem~\ref{richter} to the moment functional $\Lambda^\mu$ on $\cE:=\cH_q(E)$.
Clearly, $\dim \cE=(\dim E)\, q^2$.
Then there is an $N$\nobreakdash-atomic measure $\nu=\sum_{j=1}^N\, M_j\delta_{x_j}$, having the properties stated above, such that $\Lambda^\mu=\Lambda^\nu$.
By Proposition~\ref{equivalence} this implies that $L^\nu=L^\mu$ and the equations \eqref{estE} and \eqref{Lrepnu} follow. 
\end{proof}

The next corollary contains a bound \eqref{bound} for the number of atoms in terms of the rank of the matrix moment functional.

\begin{cor}
Suppose  $L$ is a matrix moment functional on $E$ such   that $r:=\rank L\geq1$.
Let $n:=\dim E$.
Then there exist $k\in\dN$, vectors $v_1,\dotsc,v_k\in\dC^q$, and points $x_1,\dotsc,x_k\in\cX$ such that $k\leq nr$  and $L=L^\nu$ with $\nu=\sum_{j=1}^k v_jv_j^*\, \delta_{x_j}$, that is, $L(f)=\sum_{j=1}^{k} f(x_j)v_jv_j^*$ for $f\in E$.
In particular,
\begin{equation}\label{bound}
    k
    \leq nr
    \leq n\min\{n,q^2\}
    \leq n^2.
\end{equation}
\end{cor}
\begin{proof}
Since $\dim\ran L=r\geq1$, we can chose an orthonormal basis $H_1,\dotsc,H_r$ of the linear subspace $\ran(L)$ of the real Hilbert space $(\cH_q,\langle \cdot,\cdot\rangle)$.
Then we can find a basis $f_1,\dotsc,f_n$ of $E$ such that $L(f_j)=H_j$ for  $j=1,\dotsc,r$. 
Clearly,
\[
\cE
:=\Lin\{f_jH_i\colon j=1,\dotsc,n;\,i=1,\dotsc,r\}
\]
is a linear subspace of $\cM(\cX,\fX; \cH_q)$ with $\dim\cE=nr$.
We define a linear functional  $\Lambda\colon\cE\to\dR$  by 
\begin{equation}\label{deflambda}
\Lambda(F):=\sum_{i=1}^r\langle L( \langle F,H_{i}\rangle), H_{i}\rangle, ~~ F\in \cE.
\end{equation}
Because $L$ is a matrix moment functional, there exists $\mu\in\cM_L$.
Thus, $L=L^\mu$.
Let $\tau$ be the trace measure of $\mu$ and $\rnd $ the Radon--Nikodym derivative of $\mu$ with respect to $\tau$.

We show that $\Lambda=\Lambda^\mu$.
It suffices to prove the equality $\Lambda(F)=\Lambda^\mu(F)$ for an element $F:=f_jH_i$ of $\cE$.
Since  $H_1,\dotsc,H_r$ is  an orthonormal basis, we have $\langle F,H_s\rangle=f_j\delta_{si}$  for $s=1,\dotsc,r$.
Moreover $L(f_j)=L^\mu(f_j)=\int_\cX f_j \dif\mu =\int_\cX f_j\rnd  \dif\tau$.
Using these facts we derive
\[\begin{split}
    \Lambda(F)
    &=\sum_{s=1}^r\langle L( \langle F,H_{s}\rangle), H_{s}\rangle
    =\sum_{s=1}^r\langle L(f_j\delta_{si}),H_s\rangle
    =\langle L(f_j),H_i\rangle\\
    &=\langle H_i, L(f_j)\rangle
    =\tr(H_iL(f_j))
    =\tr\left(H_i\left(\int_\cX  f_j\rnd \dif\tau\right)\right)\\ 
    &=\tr \left(\int_\cX f_jH_i \rnd \dif\tau\right)
    =\int_\cX \tr(F\rnd )\dif\tau
    =\int_\cX \langle F, \rnd \rangle\dif\tau
    =\Lambda^\mu(F),
\end{split}\]
which proves that $\Lambda=\Lambda^\mu$. 

Now we  apply Theorem~\ref{richter} to  $\Lambda^\mu$.
Then there exist vectors $v_1,\dotsc,v_k\in\dC^q$ and points $x_1,\dotsc,x_k\in\cX$ such that $k\leq nr$ and  $\nu:=\sum_{j=1}^{k}v_jv_j^*\delta_{x_j}$ belongs to $\cM_\Lambda$.
Let $f\in E$ and $i\in \{1,\dotsc,r\}$.
Using again that $H_1,\dotsc,H_r$ is an orthonormal basis \eqref{deflambda} implies that $ \langle L(f),H_i\rangle=\Lambda(fH_i)$.
Then, by \eqref{finiteat}, 
\[\begin{split}
    \langle &L(f),H_i\rangle
=\Lambda(fH_i)
    =\Lambda^\nu(fH_i)=\sum_{j=1}^k\langle f(x_j)H_i,v_jv_j^*\rangle\\
& =\sum_{j=1}^k\langle f(x_j)v_jv_j^*,H_i\rangle =\left\langle \sum_{j=1}^kf(x_j)v_jv_j^*,H_i\right\rangle
    =\langle L^\nu(f),H_i\rangle.
\end{split}\] 
Hence, $L(f)=L^\nu(f)$.
Therefore, $L=L^\nu$.
\end{proof}

Next we define with Carath\'eodory numbers important quantities of moment functionals and matrix moment functionals.

Recall that $\cM_\Lambda^\mathrm{fa}$ and $\cM_L^\mathrm{fa}$ denote the finitely atomic representing measures of a moment functional $\Lambda$ and a matrix moment functional $L$, respectively. 
Theorem~\ref{richter} and Corollary~\ref{richtercor} ensure that these sets are not empty.

\begin{dfn}\label{defat}
Let $\mu=\sum_{j=1}^k M_j\delta_{x_j}$ be a finitely atomic measure  such that $M_j\in \cH_{q,\succeq}$, $j=1,\dotsc,k$,  and $x_i\neq x_j$ for all $i\neq j$, $i,j=1,\dotsc,k$.
Define
\begin{align*}
    \atnr (\mu)&:= \#\{i\in \{1,\dotsc,k\}\colon M_i\neq O\},\\
    \rank (\mu)&:=\sum_{j=1}^k \rank(M_j).
\end{align*}
\end{dfn}

For arbitrary $\mu\in\cM_q(\cX,\fX)$, let
\[
    \ats(\mu)
    :=\{x\in\cX\colon\mu(\{x\})\neq O\}
\]
be the set of atoms of $\mu$.
Observe that in the situation of Definition~\ref{defat}, we have
\begin{align*}
    \atnr(\mu)&=\#\ats(\mu),&    
    \rank(\mu)&=\sum_{x\in\ats(\mu)}\rank(\mu(\{x\}))&
    &\text{and}&
    \atnr(\mu)&\leq\rank(\mu).
\end{align*}

\begin{dfn}\label{cara}
For a moment functional $\Lambda$, $\Lambda\neq 0$, on $\cE$ we define the \emph{Carath\'eodory numbers} $\Car (\Lambda)$ and $\car(\Lambda)$ by
\begin{align*}
    \Car(\Lambda)&:=\min\,  \{\rank(\mu)\colon \mu\in\cM_\Lambda^\mathrm{fa}\, \},\\
    \car(\Lambda)&:=\min\, \{\atnr(\mu)\colon  \mu\in\cM_\Lambda^\mathrm{fa}\, \}.
\end{align*} 
The maximum of numbers $\Car(\Lambda)$ and $\car(\Lambda)$, respectively, for $\Lambda\in \fL(\cE)$, $\Lambda\neq 0$, are denoted by $\Car(\fL(\cE))$ and $\car(\fL(\cE))$, respectively.
\end{dfn}

\begin{dfn}\label{cara1}
For a matrix moment functional $L$, $L\neq O$, on $E$ we define the \emph{Carath\'eodory numbers} $\Car (L):=\Car(\Lambda_L)$ and $\car(L):=\car(\Lambda_L)$, where $\Lambda_L$ is the moment functional on $\cE:=\cH_q(E)$ given by \eqref{defcl}.
Further, we set $\Car(\cL_q(E)):=\Car(\fL(\cE))$ and $\car(\cL_q(E)):=\car(\fL(\cE))$.
\end{dfn}

Proposition~\ref{equivalence} shows
\begin{align*}
    \Car(L)&=\min\,  \{\rank(\mu)\colon \mu\in\cM_L^\mathrm{fa}\, \},\\
    \car(L)&=\min\, \{\atnr(\mu)\colon  \mu\in\cM_L^\mathrm{fa}\, \}.
\end{align*}
Note that in the case $q=1$ the numbers $\Car(\Lambda)$, $\Car(L)$, $\car(\Lambda)$, $\car(L)$ coincide with the corresponding scalar Carath\'eodory numbers (see e.\,g.\ \cite[Definition~3]{DioS2}). 

In Definition~\ref{defat} we assumed that the points   $x_j$ are pairwise different.
Without this assumption it is easily verified that for a finitely atomic measure $\mu=\sum_{j=1}^k M_j\delta_{x_j}$, with $M_j\in \cH_{q,\succeq}$ for $j=1,\dotsc,k$, we have 
\[
    \#\ats(\mu)\leq\atnr(\mu)~~~  \text{ and } ~~~ \rank(\mu)\leq \sum_{j=1}^k\rank(M_j).
\]
Therefore, it follows from Theorem~\ref{richter} and Corollary~\ref{richtercor} that in Definitions~\ref{cara} and~\ref{cara1} we have 
\begin{align*}%
    \car(\Lambda)\leq \Car(\Lambda) &\leq \dim \cE,\\
    \car(L)\leq   \Car (L)  &\leq (\dim E)\, q^2.
\end{align*}
These inequalities  provide only rough estimates and a finer analysis is desirable.
Here the results obtained in \cite{MR4263684,DioS2,RienerS} on Carath\'eodory numbers in the scalar polynomial case might be useful.

Let us treat two examples.
We begin with a simple lemma.

\begin{lem}\label{aux}
Let $R\in \cH_{2,\succeq}$ and $S\in \cH_2$.
Suppose that $ \ker R=\{0\}$ and $S\neq xR$ for all $x\in \dR$.
Then there exists a $\lambda_0\in \dR$ such that $\lambda_0 \cdot R+S\succeq O$ and $\rank\,(\lambda_0\cdot R+S)= 1$.
\end{lem}
\begin{proof}
Consider the set $\cM:=\{\lambda \in \dR\colon \lambda \cdot R+S\succeq O \}$.
The assumptions on $R$ imply that there are numbers $C>c>0$ such that $C\cdot I\succeq R\succeq c\cdot I $. 
Hence $\cM$ is not empty and bounded from below (because $R_{11}>0$).
Obviously, $\cM$ is closed.
Let $\lambda_0=\min \cM$.
Then $\lambda_0\cdot R+S\succeq O$. 
If $\ker\, (\lambda_0 R+S)$ would be trivial, then $\lambda_0 R+S\succeq \varepsilon\cdot I$ for some $\varepsilon >0$, so $\lambda_0\cdot R+S\succeq \varepsilon C^{-1}\cdot R$ and hence $(\lambda_0- \varepsilon C^{-1})\cdot R+S\succeq O$, which contradicts $\lambda_0 =\min \cM$.
Thus, $\ker (\lambda_0 R+S)\neq\{0\}$.
Since $\lambda_0\cdot R+S\neq O$ by assumption, we have $\rank (\lambda_0\cdot R+S)=1$.
\end{proof}

\begin{exm}\label{E1311}
Suppose  $\cX=\dR$ and $E=\Lin \{1,x\}$.
Let $L\in\scrL(E,\cH_q)$ and set $R:=L(1)$ and $S:=L(x)$. 
Suppose that $R\neq O$ and $R\succeq O $.

Case 1:
$\ker R\neq \{0\}$.

Without loss of generality $u:=(1,0)^\trn \in \ker R$.

Case 1.1:
$u\notin \ker S$.

Then $L$ is not a moment functional.
Assume to the contrary that $L\in \cL_q(E)$.
By Corollary~\ref{richtercor}, $L$ has a representing measure $\mu=\sum_{j=1}^k \delta_{x_j}M_j$.
Then  $Ru=L(1)u= \sum_{j=1}^k M_ju=0$.
Since $M_j\succeq O $, it follows that $M_ju=0$ for all $j$, so $Su=L(x)u=\sum_{j=1}^k x_jM_ju=0$, which contradicts the assumption.

Case 1.2:
$u\in \ker S$.

Since $u\in \ker R, u\in \ker S$ and $R,S\in \cH_2$, the matrix entries $R_{11}$, $R_{12}$, $R_{21}$ of $R$ and $S_{11}$, $S_{12}$, $S_{21}$ of $S$ are zero. Note that $R_{22}>0$, because $R\succeq O$ and $R\neq O$ by assumption.  Then  $L$ is a moment functional with representing measure
\begin{align*}
\mu:=\delta_{S_{22}R_{22}^{-1}}
\begin{pmatrix}
O&O \\
O& R_{22}
\end{pmatrix}
\end{align*} and we obviously have $\Car(L)=\car (L)=1$.

Case 2:
$\ker R=\{0\}$.

Case 2.1: $S=x\cdot R$ for some $x\in \dR$.

Then $L(a+bx)= aR+bS = (a+bx)R $ for $a,b\in \dR$. Hence $\mu:=\delta_xR$ is a representing measure for $L$, so $L$ is a matrix moment functional. Thus, $\car(L)=1$ and $\Car(L)\leq \rank (\mu)=\rank R=2$. Since $\mu$ is the only  $1$-atomic representing measure for $L$ as easily checked, we have $\Car(L)=2$.

Case 2.2: $S\neq x\cdot S$ for all $x\in \dR$.

By the  Cases 2 and 2.2, $R$ and $S$, $-S$ satisfy the assumptions of Lemma~\ref{aux}. Let $\lambda_+$ and $\lambda_-$ denote the corresponding numbers $\lambda_0$ from Lemma~\ref{aux} for $S$ and $-S$, respectively. Set $N_+:=\lambda_-\cdot R+S$ and $N_-:=\lambda_+\cdot R-S$. From Lemma~\ref{aux} we obtain $\rank N_+=\rank N_-=1$ and $N_+\succeq O$, $N_-\succeq O$. Then $(\lambda_+ +\lambda_-)\cdot R=N_++N_-$ and hence $\lambda_++\lambda_->0$.
Set
\begin{align*}
M_+:=(\lambda_+ +\lambda_-)^{-1}\cdot N_+, ~~~M_-:=(\lambda_+ +\lambda_-)^{-1}N_-.
\end{align*}

From the equations $M_++M_-=(\lambda_+ +\lambda_-)^{-1} (N_++N_-)=R=L(1)$ and
\begin{align*}
\lambda_+\cdot M_+& +(- \lambda_-)\cdot M_- =(\lambda_+ +\lambda_-)^{-1} (\lambda_+ \cdot N_+-\lambda_-\cdot N_-)\\ &=(\lambda_+ +\lambda_-)^{-1} (\lambda_+ \cdot(\lambda_-\cdot R+S)-\lambda_-\cdot(\lambda_+\cdot R-S))=S=L(x)
\end{align*}
we conclude that $\mu:=\delta_{\lambda_+}M_+ +\delta_{-\lambda_-}M_-$ is a representing measure for $L$.
Thus, in particular, $L$ is a matrix moment functional. Since $\lambda_++ \lambda_->0$, we have $\lambda_+\neq -\lambda_-$, so $\car(L)\leq \atnr(\mu)=2$ and $\Car (L)\leq \rank(\mu)=\rank M_++ \rank M_-=2$. From the requirement $S\neq x\cdot R$ for all $x\in \dR$ in Case 2.2 it follows that $L$ has no $1$-atomic representing measure. Therefore, we have $\car(L)=2$ and $\Car(L)=2$.
\end{exm}

\begin{exm}
Let $\cX$ be the set of  points $x_1=(-1,-1)$, $x_2=(-1,1)$, $x_3=(1,-1)$, $x_4=(1,1)$ in $\dR^2$ and $E=\Lin\{x,y\}$. 
Given $R,S\in \cH_2$, we define a linear mapping $L:E\mapsto \cH_2$ by $L(x)=R, L(y)=S$.

Any measure $\nu\in \cM_2(\cX,\fX)$ is of the form $\nu=\sum_{j=1}^4 \delta_{x_j}M_j$ with $ M_j\in \cH_{2,\succeq}$.
Clearly, $\nu$ is a representing measure for $L$ if and only if
\begin{align*}
    R&=-M_1-M_2+M_3+M_4,&
    S&=-M_1+M_2-M_3+M_4,
\end{align*}
or equivalently, 
\begin{align}\label{mjs}
    R+S&=2(M_4-M_1),&
    R-S&=2(M_3-M_2).
\end{align}
Since these equations can be always satisfied,  $L$ is a matrix moment functional.
How many atoms are needed depends on the fact  whether or not $R+S$ or $R-S$ are in $\cH_{2,\succeq}$. 
We mention only one possible case.

Suppose that $\rank(R+S)=\rank (R-S)=2$ and non of the four matrices $\pm(R+S)$, $\pm( R-S)$ is positive. 
We write  $R+S$ and $R-S$ as in \eqref{mjs} with $\rank M_j=1$ for $j=1,\dotsc,4$. Then $\rank (\nu)=4$ and $\atnr(\nu)=4$. The assumptions imply that all four  matrices $M_j$ in \eqref{mjs} are needed. Hence $\Car(L)=4$ and $\car(L)=4$.
\end{exm}

\section{Strictly positive functionals}\label{strictlypositive}

\begin{dfn}
A linear functional $\Lambda$ on $\cE$  is said to be \emph{strictly $\cE_{\succeq}$-positive}
if~  $\Lambda(F)>0$\, for all $F\in \cE_{\succeq}$, $F\neq O$. 
\end{dfn}

Let us recall the following  notation:
If $C$ is a convex cone in a real vector space $V$, then  its \emph{dual cone}  $C^\wedge$ in the dual vector space $V^*$ is defined by
\[
    C^\wedge \, 
    :=\{\ell\in V^*\colon\ell(v)\geq 0\text{ for }v\in C\}.
\]

\begin{lem}\label{128}
Let $ \overline{\fL(\cE)}$ denote the closure of $\fL(\cE)$ in some norm of the (finite-dimensional) vector space $\cE^*$.
Then 
\begin{equation}\label{lqesim}
    \fL (\cE)
    \subseteq (\cE_{\succeq})^\wedge \, 
    =\overline{\fL(\cE)}
\end{equation}
and
\begin{equation}\label{128.B}
(\fL(\cE))^\wedge
=\cE_\succeq
=((\cE_\succeq)^\wedge)^\wedge.
\end{equation}
\end{lem}
\begin{proof}
The main part of proof is to show that\, $(\cE_{\succeq})^\wedge\subseteq \clo{\fL(\cE)}$.
Assume to the contrary that there exists an $\Lambda_0\in(\cE_{\succeq})^\wedge$ such that $\Lambda_0\notin\clo{\fL(\cE)}$.
Then, by the separation of convex sets, applied to the closed (!) convex set $\clo{\fL(\cE)}$ in $\cE^*$, there exists a linear functional $\varphi$ on 
$\cE^*$ 
such that $\varphi(\Lambda_0)<0$ and $\varphi(\Lambda)\geq 0$ for all $\Lambda\in\clo{\fL(\cE)}$.
Since $\cE$ is finite-dimensional,  this functional is given by some element $F_0$ of $\cE$, that is, $\varphi(\Lambda)=\Lambda(F_0)$ for all $\Lambda\in \cE^*$.

Let $x\in \cX$ and $v\in \dC^q$.
Define $\Lambda\in \cE^*$ by $\Lambda(F)=v^*F(x)v$, $F\in \cE$.
Clearly, $\Lambda$ is a moment functional on $\cE$ with representing measure $\delta_x vv^*$.
Hence $\varphi(\Lambda)\geq 0$, so that  $\varphi(\Lambda)=\Lambda(F_0)=v^*F_0(x)v\geq 0$ for all $v\in \dC^q$.
Thus $F_0(x)\succeq O$ and hence $F_0\in \cE_{\succeq}$. 
Since $\Lambda_0\in(\cE_{\succeq})^\wedge$, we therefore have $ \varphi(\Lambda_0)=\Lambda_0(F_0)\geq 0$, a contradiction.
This proves that $(\cE_{\succeq})^\wedge\subseteq \clo{\fL(\cE)}$.

From Lemma~\ref{L1036+} we know that $\fL(\cE)\subseteq (\cE_{\succeq})^\wedge$.
Hence, since $(\cE_{\succeq})^\wedge$ is obviously closed in $\cE^*$, it follows that
$\ov{\fL_q(\cE)}\subseteq (\cE_{\succeq})^\wedge$.
Putting all these relations together \eqref{lqesim} follows.

Obviously, $(\overline{\fL(\cE)})^\wedge=(\fL(\cE))^\wedge$.
From \eqref{lqesim} we get $((\cE_\succeq)^\wedge)^\wedge=(\overline{\fL(\cE)})^\wedge$.
Furthermore, since $\cE_{\succeq}$ is closed in $\cE$, the bipolar theorem (see e.\,g.\ \cite[Theorem~14.1]{MR0274683}) yields $((\cE_\succeq)^\wedge )^\wedge=\cE_\succeq$.
Putting all these relations together \eqref{128.B} follows.
\end{proof}

\begin{lem}\label{L129}
Let $\lVert\cdot\rVert_0$ and $\lVert\cdot\rVert_1$ be norms on the vector spaces $\cE$ and $\cE^*$, respectively.
For any $\Lambda\in \cE^*$ the following are equivalent:
\begin{enumerate}[(i)]
    \item\label{L129.i} $\Lambda$ is strictly $\cE_{\succeq}$-positive.
    \item\label{L129.ii} There exists a number $c>0$ such that
    \begin{equation}\label{1.12}
        \Lambda(F)\geq c\lVert F\rVert_0\quad\text{for all }F\in\cE_{\succeq}.
    \end{equation}
    \item\label{L129.iii} $\Lambda$ is an inner point of the cone $(\cE_{\succeq})^\wedge$ in the normed space $(\cE^*,\lVert\cdot\rVert_1)$.
\end{enumerate}
\end{lem}
\begin{proof}
\ref{L129.i}$\to$\ref{L129.ii}:
If $\cE_{\succeq}=\{0\}$, then \eqref{1.12} is fulfilled with $c:=1$.
Now suppose $\cE_{\succeq}\neq\{0\}$. 
Then $\cU_+:=\{ F\in \cE_{\succeq}\colon \lVert F\rVert_0=1\}$ is non-empty.
Define $c:= \inf\, \{ \Lambda(F)\colon F\in \cU_+\}$.
Obviously,  $\cE_{\succeq}$ is closed in $(\cE,\lVert\cdot\rVert_0)$.
Thus $\cU_+$ is a bounded and closed subset of the  finite-dimensional normed vector space $(\cE,\lVert\cdot\rVert_0)$.
Hence $\cU_+$ is compact.
The linear functional $\Lambda$ is continuous on 
the 
finite-dimensional space $(\cE,\lVert\cdot\rVert_0)$.
Therefore, the infimum is attained at some $F_0\in\cU_+$.
Since $F_0\in \cU_+$, we have $F_0\in \cE_{\succeq}$ and $F_0\neq O$.
Hence, by~\ref{L129.i}, $c=\Lambda(F_0)>0$.
Then, by the definition of $c$,  $\Lambda(F)\geq c$ for all $F\in\cU_+$.
By scaling $F$ this yields \eqref{1.12}. %

\ref{L129.ii}$\to$\ref{L129.iii}:
From \eqref{1.12} we obtain $\Lambda(F)\geq c\lVert F\rVert_0\geq 0$ for all $F\in\cE_{\succeq}$.
Thus, $\Lambda\in(\cE_{\succeq})^\wedge$. 
We equip the space $\cE^*$ with the  norm $\lVert\cdot\rVert$ defined by 
\[
    \lVert\Lambda'\rVert
    :=\sup~ \{\, \lvert\Lambda'(F)\rvert\colon F\in\cE,\; \lVert F\rVert_0=1\},
    \quad \Lambda'\in \cE^*.
\]

Now we take  $\Lambda'\in \cE^*$ such that $\lVert \Lambda-\Lambda'\rVert_0<c$, where $c$ is the positive constant occurring in \eqref{1.12}.
Let $F\in \cE_{\succeq}$.
Using \eqref{1.12}  we derive
\[\begin{split}
    c\lVert F\rVert_0-\Lambda'(F)
    \leq \Lambda(F)-\Lambda'(F)
    =(\Lambda-\Lambda')(F)
    \leq \lVert \Lambda-\Lambda'\rVert \, \lVert F\rVert_0
    \leq c\lVert F\rVert_0,
\end{split}\]
so that $\Lambda'(F)\geq 0$.
Hence, $\Lambda'\in(\cE_{\succeq})^\wedge$.
This shows that $\Lambda$ is an inner point of $(\cE_{\succeq})^\wedge$ in the normed  space $(\cE^*,\lVert\cdot\rVert)$.
Since all norms on a finite-dimensional vector space are equivalent,  $\Lambda$ is also an inner point of $(\cE_{\succeq})^\wedge$ in  $(\cE^*,\lVert\cdot\rVert_1)$. 

\ref{L129.iii}$\to$\ref{L129.i}:
Suppose that $F\in \cE_{\succeq}$ and $F\neq O$.
Then there exists an $x\in \cX$ such that $F(x)\neq O$.
Because of $F\in \cE_{\succeq}$, there exists a vector $v\in \dC^q$ such that $v^*F(x)v>0$.
Define $\Lambda'\in \cE^*$ by $\Lambda'(G)=v^*G(x)v$, $G\in \cE$.
By~\ref{L129.iii}, $\Lambda$ is an inner point of 
$(\cE_{\succeq})^\wedge$ in $(\cE^*,\lVert\cdot\rVert_1)$.
Hence there exists an $\epsilon >0$ such that $(\Lambda-\epsilon \Lambda')\in (\cE_{\succeq})^\wedge$.
Since $\Lambda'(F)=v^*F(x)v$, we obtain
\[
    \Lambda(F)-\epsilon \, v^*F(x)v
    = \Lambda(F)-\epsilon\, \Lambda'(F)
    = (\Lambda-\epsilon \Lambda')(F) 
    \geq 0.
\]
Consequently, $\Lambda(F)\geq \epsilon\, v^*F(x)v>0$.
Thus, $\Lambda$ is strictly $\cE_\succeq$-positive.
\end{proof}

The following theorem is the main result of this section.

\begin{thm}\label{T130}
Suppose that $\Lambda $ is a strictly $\cE_\succeq$-positive linear functional on $\cE$.
Then  $\Lambda$ is a moment functional.

Further, for any point $x\in\cX $ and matrix $M\in \cH_{q,\succeq}$,  there exists a finitely atomic measure $\nu\in\cM_{\Lambda}$ and a number $\epsilon>0$ such that $\nu(\{x\})\succeq\epsilon M$.
\end{thm}
\begin{proof}
Let $\lVert\cdot\rVert_1$ be a norm on $\cE^*$.
Since $\Lambda$ is strictly $\cE_\succeq$-positive, by Lemma~\ref{L129} it is an inner point of the cone $(\cE_{\succeq})^\wedge$ in $(\cE^*,\lVert\cdot\rVert_1)$.
Further, by Lemma~\ref{128},  $(\cE_{\succeq})^\wedge$ is the closure  $\overline{\fL(\cE)}$ of the convex cone $\fL(\cE)$ in the normed space $(\cE^*,\lVert\cdot\rVert_1)$.
The convex set $\fL(\cE)$ and its closure  $\overline{\fL(\cE)}$ have the same inner points (see e.\,g.\ \cite[Theorem~6.3]{MR0274683}).
Hence $\Lambda$ is an inner point of $\fL(\cE)$.
In particular, $\Lambda \in \fL(\cE)$ which proves the first assertion. 

We choose a norm $\lVert\cdot\rVert_0$  on $\cE$ and define $\Lambda'\in \cE^*$ by $\Lambda'(F):=\langle F(x),M\rangle$.
The linear functional $\Lambda'$ is continuous on the finite-dimensional normed space $(\cE,\lVert\cdot\rVert_0)$.
Hence there is a number $C>0$ such that $\lvert\Lambda'(F)\rvert \leq C\lVert F\rVert_0$ for $F\in \cE$.
By Lemma~\ref{L129},  there exists a $c>0$ such that $\Lambda(F)\geq c\lVert F\rVert_0$ for all $F\in \cE_{\succeq}$.
Let $\epsilon>0$ be such that $\epsilon C<c$. 
Then, for $F\in \cE_{\succeq}$, $F\neq O$, we obtain
\[
    (\Lambda-\epsilon \Lambda')(F)
    \geq c \lVert F\rVert_0- \epsilon C\lVert F\rVert_0
    =(c-\epsilon C)\lVert F\rVert_0
    >0.
\]
Thus, $(\Lambda-\epsilon \Lambda')\in \cE^*$ is strictly $\cE_{\succeq}$-positive.
Therefore, $(\Lambda-\epsilon \Lambda')\in \fL(\cE)$ by the first assertion.
By Theorem~\ref{richter}, $\Lambda-\epsilon \Lambda'$ has a finitely atomic representing measure $\nu_0$.
Then $\nu:=\nu_0+\epsilon M\delta_x$ is a finitely atomic representing measure of $\Lambda$ and $\nu(\{x\})=\nu_0(\{x\})+\epsilon M\succeq \epsilon M$, which is the second assertion. 
\end{proof}

Recall that, for a subset $\cU$ of $\cX$, we set $\cU\cap\fX:=\{\cU\cap X\colon X\in\fX\}$.
Furthermore, for a measurable space $(\cU,\fU)$, we denote by $\cM(\cU,\fU;\cH_q)$ the real vector space of measurable mappings $G\colon\cU\to\cH_q$.

\begin{cor}\label{C131P}
Suppose $\Lambda$ is a moment functional on $\cE$ and $\mu\in\cM_\Lambda$.
Let  $\cU\in\fX\setminus\{\emptyset\}$ and set $\fU:=\cU\cap\fX$.
For all $x\in\cU$, let $\bU_x$ be a linear subspace of\, $\dC^q$ and denote the orthogonal projection on $\bU_x$ by $P(x)$.
Suppose  $\mu(\cX\setminus\mathcal{U})=O$,  $\ran (\mu(\{x\}))\subseteq\bU_x$ for all $x\in\cU$, $P(x)$ belongs to $\cM(\cU,\fU;\cH_q)$, and  the following condition holds:    
\begin{enumerate}[$(\ast)$]
    \item\label{C131P.I} If $F\in\cE$ satisfies $P(x)F(x)P(x)\succeq O$ for all $x\in\cU$ and $ \Lambda(F)=0$, then $P(x)F(x)P(x)=O$ for all $x\in\cU$.
\end{enumerate}
Then, for each $\xi\in\cU$ and  $M\in\cH_{q,\succeq} $ with $\ran (M)\subseteq\bU_\xi$, there exist a $\nu\in\cM^\mathrm{fa}_\Lambda$ and an $\epsilon\in(0,\infty)$ such that $\nu(\{\xi\})\succeq\epsilon M$.
\end{cor}
\begin{proof}
Because of $\cU\in\fX$,  $\mathfrak{U}$ is  a sub-$\sigma$-algebra of $\fX$ on $\cU$, so that $(\cU,\mathfrak{U})$ is a non-trivial measurable space, and $\tilde{\mu}:= \mu\lceil\fU $ belongs to $\cM_q(\cU,\mathfrak{U})$.
As usual, let $\tau$ be the trace measure of $\mu$ and $\rnd $   the Radon--Nikodym derivative of $\mu$ with respect to $\tau$.
Then $\tilde\tau:=\tau\lceil\fU $ is the trace measure of $\tilde\mu$ and $\tilde \rnd :=\rnd \lceil\cU $ is  the Radon--Nikodym derivative of $\tilde\mu$ with respect to $\tilde\tau$.
For $F\in \cE$ we set $\tilde{F}:= F\lceil\cU $.
Since $\ran (\mu(\{x\}))\subseteq\bU_x$ for $x\in\cU$, we have $P\tilde \rnd  P=\tilde{\rnd }$.
Since $\mu\in\cM_\Lambda$ and $\mu(\cX\setminus\mathcal{U})=O$, we obtain 
\begin{equation}\label{C131P.1}\begin{split}
    \Lambda(F)
    &=\Lambda^\mu(F)
    =\int_{\cX }\langle F,\rnd \rangle\dif\tau
    =\int_{\mathcal{U} }\langle F,\rnd \rangle\dif\tau
    =\int_{\mathcal{U} }\langle\tilde{F},\tilde{\rnd }\rangle\dif\tilde{\tau}\\
    &=\int_{\mathcal{U} }\langle\tilde{F},P\tilde{\rnd }P\rangle\dif\tilde{\tau}
    =\int_{\mathcal{U} }\langle P\tilde{F}P,\tilde{\rnd }\rangle\dif\tilde{\tau}
    \qquad\text{for all }F\in\cE.
\end{split}\end{equation}
Then, $\tilde{\cE}:= \{P\tilde{F}P\colon F\in\cE\}$ is a linear subspace of $L^1(\tilde{\mu};\cH_q)$ and $\dim \tilde{\cE}\leq \dim  \cE$, because each basis of $\cE$ gives rise to a spanning set of $\tilde{\cE}$.
In particular, $\tilde{\cE}$ is finite-dimensional.
Obviously, $\tilde{\Lambda}:=\Lambda^{\tilde{\mu}}$ belongs to $\fL(\tilde{\cE})$ and $\tilde{\mu}\in\cM_{\tilde{\Lambda}}$.
Therefore, by Lemma~\ref{L1036+}, $\tilde{\Lambda}(G)\geq0$ for  $G\in\tilde{\cE}_\succeq$.

Take an  $G\in\tilde{\cE}_\succeq$, $ G\neq O$.
Then $G\in\tilde{\cE}$ and $G(x)\succeq  O$ for all $x\in\mathcal{U} $.
Hence, there exists $F\in\cE$ such that $G=P\tilde{F}P$.
For all $x\in\mathcal{U} $, then $P(x)F(x)P(x)=(P\tilde{F}P)(x)=G(x)$.
In particular, $P(x)F(x)P(x)\succeq O$ for all $x\in\mathcal{U} $.
By \eqref{C131P.1}, 
\[ 
    \Lambda(F)
    =\int_{\mathcal{U} }\langle P\tilde{F}P,\tilde{\rnd }\rangle\dif\tilde{\tau}
    =\tilde{\Lambda}(P\tilde{F}P)
    =\tilde{\Lambda}(G).
\]
Now assume that $\tilde{\Lambda}(G)=0$.
Then $ \Lambda(F)=0$ and condition~($\ast$) yields $P(x)F(x)P(x)=O $ for all $x\in\mathcal{U} $.
Thus $ G(x)=O$ for all $x\in\cU$, contradicting $ G\neq O$.
Consequently, $\tilde{\Lambda}(G)\neq0$, so $\tilde{\Lambda}(G)>0$.
This proves that $\tilde{\Lambda}\in\fL(\tilde{\cE})$ is strictly $\tilde{\cE}_\succeq$\nobreakdash-positive.

Let $\xi\in\cU$ and $M\in\cH_{q,\succeq} $ with $\ran (M)\subseteq\bU_\xi$.
We apply Theorem~\ref{T130} to the measurable space $(\cU,\mathfrak{U})$, the subspace $\tilde{\cE}$ of $L^1(\tilde{\mu};\cH_q)$, and the strictly $\tilde{\cE}_\succeq$\nobreakdash-positive functional $\tilde{\Lambda}$. %
Then  there exist $\tilde{\nu}\in\cM^\mathrm{fa}_{\tilde{\Lambda}}$ and $\epsilon\in(0,\infty)$ such that $\tilde{\nu}(\{\xi\})\succeq\epsilon M$.
Set $\xi_0:=\xi$.
According to Theorem~\ref{richter}, applied to the moment functional $\tilde\Lambda$ on $\tilde{\cE}$,  there are points $\xi_1,\dotsc,\xi_n\in\mathcal{U} $ and matrices $\tilde{M}_0,\tilde{M}_1,\dotsc,\tilde{M}_n\in\cH_{q,\succeq} $, $n\in \dN_0$,  with $\tilde{M}_0\succeq\epsilon M$,  such that $\tilde{\nu}:= \sum_{j=0}^n\tilde{M}_j\tilde{\delta}_{\xi_j}$, where the point measures $\tilde{\delta}_{\xi_j}$ are defined on $\mathfrak{U}$.
Then $M_j:=P(\xi_j)\tilde{M}_jP(\xi_j)\in\cH_{q,\succeq}$ and $\ran (M_j)\subseteq\ran (P(\xi_j))=\bU_{\xi_j}$.
Set $\nu:= \sum_{j=0}^n M_j\delta_{\xi_j}$, where $\delta_{\xi_0},\delta_{\xi_1},\dotsc,\delta_{\xi_n}$ are defined on $\fX$.
Then $\nu\in\cM^\mathrm{fa}_q(\cX,\fX)$.
Using \eqref{C131P.1} and Example~\ref{E1700}  we derive 
\[\begin{split}
    \Lambda(F)
    &=\int_{\mathcal{U} }\langle P\tilde{F}P,\tilde{\rnd}\rangle\dif\tilde{\tau}
    =\tilde{\Lambda}(P\tilde{F}P)
    =\Lambda^{\tilde{\nu}}(P\tilde{F}P)
    =\sum_{j=0}^n \langle(P\tilde{F}P)(\xi_j),\tilde{M}_j\rangle\\
    &=\sum_{j=0}^n\langle P(\xi_j)\tilde{F}(\xi_j)P(\xi_j),\tilde{M}_j\rangle
    =\sum_{j=0}^n\langle \tilde{F}(\xi_j),P(\xi_j)\tilde{M}_jP(\xi_j)\rangle\\
    &=\sum_{j=0}^n\langle F(\xi_j),M_j\rangle
    =\Lambda^\nu(F)
\end{split}\]
for all $F\in \cE$.
Therefore, $\nu\in\cM^\mathrm{fa}_\Lambda$.
Because of $M\in\cH_{q,\succeq} $ with $\ran (M)\subseteq\bU_\xi$, we have $P(\xi)MP(\xi)=M$.
Hence, since $\xi=\xi_0$, we obtain
\[\begin{split}
    \nu(\{\xi\})
    &=\nu(\{\xi_0\})
    =M_0
    =P(\xi_0)\tilde{M}_0P(\xi_0)
    \succeq P(\xi_0)(\epsilon M)P(\xi_0)
    =\epsilon M.\qedhere
\end{split}\]
\end{proof}

\section{The set of atoms}\label{setofatoms}
Throughout this section, $\Lambda$ is a moment functional on $\cE$. 
Recall that $\ats(\mu):=\{x\in\cX\colon\mu(\{x\})\neq O\}$ denotes the set of atoms of $\mu\in\cM_q(\cX,\fX)$.

\begin{dfn}\label{D133}
The \emph{set of atoms} of $\Lambda$ is 
\begin{equation}\label{W}
    \cW(\Lambda)
    := \bigcup_{\mu\in\cM_\Lambda}\ats (\mu).
\end{equation}
Furthermore, for  $x\in\cX$, we define
\begin{equation}\label{WW}
    \bW(\Lambda;x)
    :=\bigcup_{\mu\in\cM_\Lambda}\ran (\mu(\{x\})).
\end{equation}
\end{dfn}

The following two lemmas show that we can restrict ourselves to finitely atomic measures.

\begin{lem}\label{L1054}
Let $\mu\in\cM_\Lambda$ and $x\in\cX$.
There exists $\nu\in\cM^\mathrm{fa}_\Lambda$ such that $\nu(\{x\})=\mu(\{x\})$.
\end{lem}
\begin{proof}
Let $M:=\mu(\{ x \})$ and  $\mathcal{Y}:=\cX\setminus\{ x \}$.
Then $\mathcal{Y}\in\fX$ and $\mu':=\mu-M\delta_ x $ belongs to $\cM_q(\cX,\fX)$ and satisfies $\mu'(\cX\setminus\mathcal{Y})= O$.
Further, $\Lambda':=\Lambda-\langle\ell_x(\cdot),M\rangle$ is also a moment functional on $\cE$ and $\mu'\in\cM_{\Lambda'}$.
We apply Theorem~\ref{richter} to $\Lambda'$ and obtain a measure $\nu'\in\cM^\mathrm{fa}_{\Lambda'}$ such that $\ats (\nu')\subseteq\mathcal{Y}$, so that $\nu'(\{ x \})=O$.
Consequently, $\nu:=\nu'+M\delta_ x $ is in $\cM^\mathrm{fa}_\Lambda$ and  $\nu(\{ x \})=M$.
\end{proof}

\begin{lem}\label{L1459}
$\cW(\Lambda)=\bigcup_{\nu\in\cM^\mathrm{fa}_\Lambda}\ats (\nu)$\, and\, $\bW(\Lambda;x)=\bigcup_{\nu\in\cM^\mathrm{fa}_\Lambda}\ran (\nu(\{x\}))$ for all $x\in\cX$.
\end{lem}
\begin{proof}
Clearly,
\begin{equation}\label{L1459.1}
    \bigcup\nolimits_{\nu\in\cM^\mathrm{fa}_\Lambda}~ \ats (\nu)
    \subseteq\bigcup\nolimits_{\mu\in\cM_\Lambda}~ \ats (\mu)
    =\cW(\Lambda).
\end{equation}
Let $\xi\in \cW(\Lambda)$.
By \eqref{W},  there exists $\mu\in\cM_\Lambda$ such that $\xi\in\ats (\mu)$.
Then, by definition,  $\mu(\{\xi\})\neq O$.
Lemma~\ref{L1054} yields the existence of $\nu\in\cM^\mathrm{fa}_\Lambda$ such that $\nu(\{\xi\})=\mu(\{\xi\})$.
Hence $\xi\in\ats (\nu)$.
Combined with \eqref{L1459.1} it follows that $\cW(\Lambda)=\bigcup_{\nu\in\cM^\mathrm{fa}_\Lambda}\ats (\nu)$.

Let $x\in\cX$.
From \eqref{WW} we infer
\begin{equation}\label{L1459.2}
    \bigcup\nolimits_{\nu\in\cM^\mathrm{fa}_\Lambda}~ \ran (\nu(\{x\}))
    \subseteq\bigcup\nolimits_{\mu\in\cM_\Lambda}~ \ran (\mu(\{x\}))
    =\bW(\Lambda;x).
\end{equation}
Now let $v\in\bW(\Lambda;x)$.
Then there exists $\mu\in\cM_\Lambda$ such that $v\in\ran (\mu(\{x\}))$.
By Lemma~\ref{L1054}, there is $\nu\in\cM^\mathrm{fa}_\Lambda$ such that $\nu(\{x\})=\mu(\{x\})$.
Hence $v\in\ran (\nu(\{x\}))$.
Combined with \eqref{L1459.2} we get $\bW(\Lambda;x)=\bigcup_{\nu\in\cM^\mathrm{fa}_\Lambda}\ran (\nu(\{x\}))$.
\end{proof}

\begin{lem}\label{R0742}
For all $x\in\cX$, $\bW(\Lambda;x)$ is a linear subspace of $\dC^q$ and there exists 
$\nu\in\cM^\mathrm{fa}_\Lambda$ such that $\bW(\Lambda;x)=\ran (\nu(\{x\}))$.
\end{lem}
\begin{proof}
By definition, $\bW(\Lambda;x)\subseteq\dC^q$.
Because of $\Lambda\in\fL(\cE)$, there exists $\sigma\in\cM_\Lambda$.
By \eqref{WW}, $0\in\ran (\sigma(\{x\}))\subseteq\bW(\Lambda;x)$, so that $\bW(\Lambda;x)\neq\emptyset$.
Let $\lambda\in\dR$ and  $u\in\bW(\Lambda;x)$.
Again by  \eqref{WW},  there exists $\mu\in\cM_\Lambda$ such that $u\in\ran (\mu(\{x\}))$.
Consequently, $\lambda u\in\ran (\mu(\{x\}))\subseteq\bW(\Lambda;x)$.
Now let $u,v\in\bW(\Lambda;x)$.
Then there exist $\mu,\nu\in\cM_\Lambda$ such that $u\in\ran (\mu(\{x\}))$ and $v\in\ran (\nu(\{x\}))$.
Clearly, $\sigma:=\frac{1}{2}(\mu+\nu)$ belongs to $\cM_\Lambda$ and fulfills $2\sigma(\{x\})\succeq\mu(\{x\})\succeq O$ as well as $2\sigma(\{x\})\succeq\nu(\{x\})\succeq O$.
In particular, $\ran (\mu(\{x\}))\subseteq\ran (\sigma(\{x\}))$ and $\ran (\nu(\{x\}))\subseteq\ran (\sigma(\{x\}))$.
Consequently, $u+v\in\ran (\mu(\{x\}))+\ran (\nu(\{x\}))\subseteq\ran (\sigma(\{x\}))\subseteq\bW(\Lambda;x)$.
This proves that $\bW(\Lambda;x)$ is a linear subspace of $\dC^q$. 

Let us take vectors $v_1,\dotsc,v_n\in\bW(\Lambda;x)$, $n\in \dN$, such that $\bW(\Lambda;x)=\Lin \{v_1,\dotsc,v_n\}$.
Because of Lemma~\ref{L1459}, for $j\in\{1,\dotsc,n\}$, there exists $\nu_j\in\cM^\mathrm{fa}_\Lambda$ such that $v_j\in\ran (\nu_j(\{x\}))$.
Then $\nu:=\frac{1}{n}(\nu_1+\dotsb+\nu_n)$ belongs to $\cM^\mathrm{fa}_\Lambda$ and $n\nu(\{x\})\succeq\nu_j(\{x\})\succeq O$ for all  $j$.
Therefore, $v_j\in\ran (\nu_j(\{x\}))\subseteq\ran (\nu(\{x\}))$ for $j\in\{1,\dotsc,n\}$.
Consequently, $\Lin \{v_1,\dotsc,v_n\}\subseteq\ran (\nu(\{x\}))$.
Hence $\bW(\Lambda;x)\subseteq\ran (\nu(\{x\}))$.
Since $\ran (\nu(\{x\}))\subseteq\bW(\Lambda;x)$ by definition, we have proved that $\bW(\Lambda;x)=\ran (\nu(\{x\}))$.
\end{proof}

Let $L$ be a matrix moment functional on $E$.
Then, as it is easily seen, $\tr L$ is a scalar moment functional on $E$ and $\tr\mu\in\cM_{\tr L}$ for each $\mu\in\cM_L$.
By Proposition~\ref{equivalence}, the associated functional $\Lambda_L$ is a moment functional on $\cE:=\cH_q(E)$ and $\cM_{\Lambda_L}=\cM_L$.
The following  example shows that the sets of atoms of $\tr L$ and $\Lambda_L$ are different in general.

\begin{exm}\label{E0635}
We consider the case $q=2$ and $\mathcal{X} =\dR $.
Let $E:= \Lin_\dR \{x^0,x^1,x^2\}$.
We define a linear mapping $L\colon E\to\cH_2$ by $L:=e_{11}\ell_{1}+e_{22}\ell_{-1}$, that is, 
\begin{equation}\label{defLE}
    L(f)=\left(\begin{smallmatrix}f(1)&0\\0&f(-1)\end{smallmatrix}\right) \quad  \textrm{for }f\in E.
\end{equation}
Then 
\begin{equation}\label{deflambdaL}
    \Lambda_L(F)= \tr e_{11}F(1)+\tr e_{22}F(-1)= f_{11}(1)+f_{22}(-1),\quad F=(f_{ij})\in \cH_2(E).
\end{equation}

Define $\cW(L):=\bigcup_{\mu\in\cM_L}\ats (\mu)$.
Since $\cM_L=\cM_{\Lambda_L}$ by Proposition~\ref{equivalence}, we have $\cW(L)=\cW(\Lambda_L)$.
We show that 
\begin{equation}\label{equWL}
    \cW(L)
    =\cW(\Lambda_L)
    =\{1,-1\}.
\end{equation} 
By the definition of $L$, we have  $e_{11}\delta_{1}+e_{22}\delta_{-1}\in \cM_L$, so that $1,-1\in \cW(L)$. 
Let $\mu$ be an arbitrary representing measure of $\Lambda_L$.
Then, by \eqref{deflambdaL},  
\[
    0
    =\Lambda_L((x-1)^2e_{11})
    =\int_\dR \langle (x-1)^2 e_{11},\Phi\rangle \dif \tau
    =\int_\dR (x-1)^2 \phi_{11}(x) \dif \tau.
\]
Hence $\phi_{11}(x)=0$\enspace$\tau$\nobreakdash-a.\,e.\ on $\dR\setminus\{1\}$.
Similarly, replacing $(x-1)^2e_{11}$ by $(x+1)^2e_{22}$, we obtain $\phi_{22}(x)=0$\enspace$\tau$\nobreakdash-a.\,e.\ on $\dR\setminus\{-1\}$.
Therefore, since $\Phi(x)\succeq O$ on $\dR$, we get $\Phi(x)=0$\enspace$\tau$\nobreakdash-a.\,e.\ on $\dR\setminus\{1,-1\}$.
Thus $\mu(\dR\setminus\{1,-1\})=O$, so the atoms of $\mu$ are contained in the set $\{1,-1\}$.
This proves \eqref{equWL}.
(In fact, $e_{11}\delta_{1}+e_{22}\delta_{-1}$ is the unique representing measure of $\Lambda_L$.)

Now let $\xi\geq1$.
We compute that $\tr L=\xi^{-2}\ell_{-\xi}+2(1-\xi^{-2})\ell_0+\xi^{-2}\ell_{\xi}$ on $E$.
Hence, $\xi^{-2}\delta_{-\xi}+2(1-\xi^{-2})\delta_{0}+\xi^{-2}\delta_{\xi}$ is a representing measure of the moment functional $\tr L$ on $E$.
In particular, $-\xi, 0,\xi$ belong to the set $\cW(\tr L):=\bigcup_{\mu\in\cM_{\tr L}}\ats (\mu)$.

Thus, by the preceding, we have shown that $\cM_{\tr L}\neq\{\tr\mu\colon\mu\in\cM_L\}$ and $\cW(L)\subsetneqq\cW(\tr L)$.
\end{exm}

In the next section we  show that the set $\cW(\Lambda)$ coincides  with the core set.
The following are  preliminary notions and results to achieve this goal. 

\begin{dfn}\label{D139+}
Suppose that $P\colon\cX\to\cH_q$ is a mapping such that $P(x)$ is an orthogonal projection for all $x\in\cX$.
Then we {define}
\begin{align}
    \cN_+(\Lambda,P)&:=\{F\in\ker(\Lambda)\colon(PFP)(x)\succeq O\text{ for all }x\in\cX\}\notag\\%
    \cV_+(\Lambda,P)&:=\bigcap_{F\in\cN_+(\Lambda,P)}\{x\in \cX:\det(PFP+P^\bot)=0\}\label{VggP}.
\end{align}
Furthermore, for  $x\in\cX$, we set
\begin{equation}\label{VV+P}
    \bV_+(\Lambda,P;x)
    :=\bigcap_{F\in\cN_+(\Lambda,P)}\ker(PFP+P^\bot)(x).
\end{equation}
\end{dfn}

Note that $P^\bot:=I-P$.
For simplicity, the set $\{x\in\cX\colon\det(PFP+P^\bot)(x)=0\}$ in \eqref{VggP} is  abbreviated by $\{\det(PFP+P^\bot)=0\}$.
A similar notation will be used for other sets as well.

Clearly, $\cN_+(\Lambda,P)$ is a convex cone in $\cE$ and $\bV_+(\Lambda,P;x)$ is a linear subspace of $\dC^q$ for $x\in \cX$.

\begin{lem}\label{P123P}
Suppose that $\mu\in\cM_\Lambda$, $\xi\in\ats (\mu)$, and $F\in\ker\Lambda$.
Let  $\mathcal{V}\in\fX$ be a set such that $\mu(\cX\setminus\mathcal{V})=O$.
For each $x\in\cX$, let $\bV_x$ be a linear subspace of $\dC^q$ and let $P(x)$  denote the orthogonal projection  on $\bV_x$.
Suppose that $(PFP)(x)\succeq O$ and $\ran (\mu(\{x\}))\subseteq\bV_x$ for all $x\in\mathcal{V}$.
Then,
\begin{align*}
    \ran (\mu(\{\xi\}))&\subseteq\ker(PFP+P^\bot)(\xi)&
    &\text{and}&
    \det(PFP+P^\bot)(\xi)&=0.
\end{align*}
\end{lem}
\begin{proof}
Let  $\rnd $  be the Radon--Nikodym matrix of $\mu$ with respect to  the trace measure $\tau$ of $\mu$.
Since $\ran (\mu(\{x\}))\subseteq\bV_x$, we have $(P\rnd P)(x)=\rnd (x)$ for all $x\in\mathcal{V}$.
Hence,
\[
    \langle F(x),\rnd (x)\rangle
    =\langle F(x),(P\rnd P)(x)\rangle
    =\langle(PFP)(x),\rnd (x)\rangle.
\]
Since $(PFP)(x)\succeq O$ by assumption and $\rnd (x)\succeq O$ by our convention, $\langle(PFP)(x),\rnd (x)\rangle\geq0$ by \eqref{scalarpos} for $x\in\mathcal{V}$.
From $\mu(\cX\setminus\mathcal{V})=O$ and $\xi\in\ats (\mu)$ we  infer $\tau(\cX\setminus\mathcal{V})=0$ and $\xi\in\mathcal{V}$.
We derive
\[\begin{split}
    0
    &=\Lambda(F)
    =\int_\cX\langle F,\rnd \rangle\dif\tau
    =\int_\mathcal{V}\langle F,\rnd \rangle\dif\tau
    =\int_\mathcal{V}\langle PFP,\rnd \rangle\dif\tau\\
    &\geq\langle(PFP)(\xi),\rnd (\xi)\rangle\tau(\{\xi\})
    \geq0,
\end{split}\]
implying $\langle(PFP)(\xi),\mu(\{\xi\})\rangle=0$.
Hence $(PFP)(\xi)\mu(\{\xi\})=  O $ by Lemma~\ref{tracezero}.
Since $\ran (\mu(\{\xi\}))\subseteq\bV_\xi$ by assumption, we get $P(\xi)\mu(\{\xi\})=\mu(\{\xi\})$.
Using that $P(\xi)$ is an orthogonal projection the latter implies that $(PFP+P^\bot)(\xi)\mu(\{\xi\})=  O $.
Hence $\ran (\mu(\{\xi\}))\subseteq\ker(PFP+P^\bot)(\xi)$.
Since $\mu(\{\xi\})\neq  O $, we have $\ran (\mu(\{\xi\}))\neq\{0\}$, so that $\ker(PFP+P^\bot)(\xi)\neq\{0\}$ and therefore $\det(PFP+P^\bot)(\xi)=0$.
\end{proof}

\begin{prop}\label{P140iP}
Let $\mu\in\cM_\Lambda$.
For $x\in\cX$, let $\bV_x$ and  $P(x)$ be as in Lemma~\ref{P123P}.
Suppose that $\ran (\mu(\{x\}))\subseteq\bV_x$ for $x\in\cX$.
Then we have $\ats (\mu)\subseteq\cV_+(\Lambda,P)$ and $\ran (\mu(\{x\}))\subseteq\bV_+(\Lambda,P;x)$ for all $x\in\cX$.
\end{prop}
\begin{proof}
Let $F\in\cN_+(\Lambda,P)$.
Then $F\in\ker\Lambda$ and $(PFP)(x)\succeq O$ for all $x\in\cX$.
Thus,  Lemma~\ref{P123P} applies with $\mathcal{V}=\cX$  and yields $\ran (\mu(\{\xi\}))\subseteq\ker(PFP+P^\bot)(\xi)$ for  $\xi\in\ats (\mu)$ and $\ats (\mu)\subseteq\{\det(PFP+P^\bot)= 0\}$.
Since $F\in\cN_+(\Lambda,P)$ was arbitrary, 
\[
    \ran (\mu(\{\xi\}))
    \subseteq\bigcap_{F\in\cN_+(\Lambda,P)}\ker(PFP+P^\bot)(\xi)
    =\bV_+(\Lambda,P;\xi)
\]
for all $\xi\in\ats (\mu)$ and
\[
    \ats (\mu)
    \subseteq\bigcap_{F\in\cN_+(\Lambda,P)}\{\det(PFP+P^\bot)=0\}
    =\cV_+(\Lambda,P).
\]
Therefore, since  $\mu(\{x\})=O$ for $x\in\cX\setminus\ats (\mu)$,     $\ran (\mu(\{x\}))\subseteq\bV_+(\Lambda,P;x)$.
\end{proof}

\begin{prop}\label{P144P}
Let $P\colon\cX\to\cH_q$ be a mapping such that $P(x)$ is an orthogonal projection  for all $x\in\cX$.
There exists $S\in\cN_+(\Lambda,P)$ such that
\begin{align}
    \cV_+(\Lambda,P)
    &=\{\det(PSP+P^\bot)=0\},\label{vlp1}\\
    \bV_+(\Lambda,P;x)
    &=\ker(PSP+P^\bot)(x)\quad \text{for }x\in\cX.\label{vlp2}
\end{align}
\end{prop}
\begin{proof}
In the case $\cN_+(\Lambda,P)=\{O\}$ the assertion is trivial.
Assume now that $\cN_+(\Lambda,P)\neq\{O\}$. Then there exist $F_1,\dotsc,F_k\in\cN_+(\Lambda,P)$, $k\in \dN$, such that $\{F_1,\dotsc,F_k\}$ is a maximal linear independent subset of $\cN_+(\Lambda,P)$.
We prove that $S:=  F_1+\dotsb+F_k$ has the desired properties.

Since $PSP=PF_1P+\dotsb+PF_kP$ and $(PF_jP)(x)\succeq O$ by the definition of $\cN_+(\Lambda,P)$, we have
\begin{equation}\label{P144q.1}
    (PSP)(x)
    \succeq(PF_jP)(x)
    \succeq   O
    \quad\text{for all }j\in\{1,\dotsc,k\}\text{ and }x\in\cX.
\end{equation}
Obviously, $ \Lambda(S)= \Lambda(F_1)+\dotsb+ \Lambda(F_k)=0$.
This shows that $S\in\cN_+(\Lambda,P)$.
By definition, we have\, $\cV_+(\Lambda,P)\subseteq\{\det(PSP+P^\bot)=0\}$ and $\bV_+(\Lambda,P;x)\subseteq\ker(PSP+P^\bot)(x)$ for all $x\in\cX$.

Now  consider an arbitrary $F\in\Lin(\cN_+(\Lambda,P))$.
By the choice of $F_1,\dotsc,F_k$, there exist $\lambda_1,\dotsc,\lambda_k\in\dR$ such that $F=\lambda_1F_1+\dotsb+\lambda_kF_k$.
Consequently, $PFP=\lambda_1(PF_1P)+\dotsb+\lambda_k(PF_kP)$.
Let $x\in\cX$.
For all $j\in\{1,\dotsc,k\}$, from \eqref{P144q.1} we  infer $\ker(PSP)(x)\subseteq\ker(PF_jP)(x)$, so that $\ker(PSP)(x)\subseteq\ker(PFP)(x)$.
Let $v\in\ker(PSP+P^\bot)(x)$.
Then $[(PSP)(x)]v=-[P^\bot(x)]v$, so  $[(PSP)(x)]v=0$ and $[P^\bot(x)]v=0$.
Hence, $[(PFP)(x)]v=0$ and  $[(PFP+P^\bot)(x)]v=0$, i.\,e., $v\in\ker(PFP+P^\bot)(x)$.
Thus,  $\ker(PSP+P^\bot)(x)\subseteq\ker(PFP+P^\bot)(x)$.
Therefore, \[\ker(PSP+P^\bot)(x)\subseteq\bigcap_{F\in\Lin(\cN_+(\Lambda,P))}\ker(PFP+P^\bot)(x)\subseteq\bV_+(\Lambda,P;x).\]

Suppose that $\det(PSP+P^\bot)(x)=0$.
Then $\{0\}\neq\ker(PSP+P^\bot)(x)\subseteq\ker(PFP+P^\bot)(x)$ which implies that $\det(PFP+P^\bot)(x)=0$.
Therefore, $x\in\bigcap_{F\in\Lin(\cN_+(\Lambda,P))}\{\det(PFP+P^\bot)=0\}\subseteq\cV_+(\Lambda,P)$.

Putting the preceding together the equalities \eqref{vlp1} and \eqref{vlp2} follow.
\end{proof}

\section{The core set of a moment functional}\label{coreset}
In this section, $\Lambda$ denotes a moment functional on $\cE$ such that $\Lambda\neq 0$.
Our aim  is to introduce the core set $\cV(\Lambda)$ (see \eqref{1.24}) and derive its basic properties. 

For a linear subspace $U$ of the unitary space $\dC^q$, let $\dP_U$ denote the orthogonal projection onto $U$.

\begin{dfn}\label{D147}
Set $\cV_0(\Lambda):= \cX$ and  $\bV_0(\Lambda;x):=\dC^q$ for all $x\in\cX$.
For $j\in\dN$, we define recursively
\begin{equation}\label{1.22}
    \cN_j(\Lambda)
    := \{F\in\ker\Lambda\colon(P_{j-1}FP_{j-1})(x)\succeq  O\text{ for all }x\in\cV_{j-1}(\Lambda)\}
\end{equation}
and
\begin{gather}
\cV_j(\Lambda)
    :=\bigcap_{F\in\cN_j(\Lambda)}\{x\in\cV_{j-1}(\Lambda)\colon\det(P_{j-1}FP_{j-1}+P_{j-1}^\bot)(x)=0\},\label{1.23}\\   
    \bV_j(\Lambda;x)
    :=\bigcap_{F\in\cN_j(\Lambda)}\ker\, (P_{j-1}FP_{j-1}+P_{j-1}^\bot)(x)\quad\text{for  }x\in\cV_{j-1}(\Lambda),\label{VVj}
\end{gather}
where $P_j\colon\cV_j(\Lambda)\to\cH_q$  is defined by $P_j(x):=\dP_{\bV_j(\Lambda;x)}$ for $j\in \dN_0$.
\end{dfn}

The mapping $P_j$ is well defined, since $\cV_{j}(\Lambda)\subseteq\cV_{j-1}(\Lambda)$, according to \eqref{1.23}.

\begin{dfn}\label{cordef}
The \emph{core set} of $\Lambda$ is 
\begin{equation}\label{1.24}
    \cV( \Lambda)
    := \bigcap_{j=0}^\infty \cV_j(\Lambda).
\end{equation}
Furthermore, we set
\[
    \bV( \Lambda;x)
    :=\bigcap_{j=0}^\infty \bV_j(\Lambda;x)
    \quad \text{for }x\in\cV( \Lambda).%
\]
\end{dfn}
Observe that $\bV( \Lambda;x)$ is well defined, since $\cV_{j}(\Lambda)\subseteq\cV( \Lambda)$.

The next example is the counter-part of Example~\ref{E0635} for  core sets.
It shows that for a matrix moment functional $L$ the core sets of  $\tr L$ and  $\Lambda_L$ are different in general.

\begin{exm}[Example~\ref{E0635} continued]\label{E1114}
We retain all notations from Example~\ref{E0635}.
Recall that $L(f)=e_{11}f(1)+e_{22}f(-1)$ for $f\in  E:= \Lin_\dR \{x^0,x^1,x^2\}$ by \eqref{defLE} and $\Lambda_L(F)=f_{11}(1)+f_{22}(-1)$ for $F=(f_{ij})\in \cH_2(E)$ by \eqref{deflambdaL}.

Clearly, $\ker(L)=\{f\in E\colon f(-1)=f(1)=0\}=\Lin_\dR \{x^2-1\}$ and the space $E_+:=\{f\in E\colon f(x)\geq0~ \text{ for}~ x\in\cX\}$ is
\begin{equation}\label{E+}
    E_+
    =\{a(x-b)^2+c\colon a,c\geq 0, \, b\in\dR\}.
\end{equation}
Set $N_1(L):=\{f\in E_+\colon L(f)=O\}$ and $V_1(L):=\{x\in\cX\colon f(x)=0\;\text{for}\;f\in N_1(L)\}$.
Then we have $N_1(L)=\ker(L)\cap E_+=\{0\}$ and $V_1(L):=\bigcap_{f\in N_1(L)}\{f=0\}=\dR$.

We determine the core set of the scalar (!) moment functional $\tr L$ on $E$.
First we note that $\ker(\tr L)=\{f\in E\colon f(-1)+f(1)=0\}=\Lin_\dR \{x^1,x^2-1\}$.
Therefore, $\cN_1(\tr L)=\ker(\tr L)\cap E_+=\{0\}$ and $\cV_1(\tr L):=\bigcap_{f\in \cN_1(\tr L)}\{f=0\}=\dR$.
Hence $\tr L$ is strictly positive on $E$ and the core set $\cV(\tr L)$ of $\tr L$ is $\dR$.

Next we  show that $\cV_1(\Lambda_L)=\{-1, 1\}$.
Indeed, we have
\[
    \mathcal{H}_2(E)
    =\left\{\left(\begin{smallmatrix}f_{11}&f_{12}-\ii f_{21}\\f_{12}+\ii f_{21}&f_{22}\end{smallmatrix}\right)\colon f_{11},f_{12},f_{21},f_{22}\in E\right\}
\]
and therefore
\begin{equation}\label{H>}
    \mathcal{H}_{2,\succeq}(E)
    =\left\{\begin{pmatrix}f_{11}&f_{12}-\ii f_{21}\\f_{12}+\ii f_{21}&f_{22}\end{pmatrix}\colon\begin{gathered}f_{11},f_{22}\in E_+\text{ and }f_{12},f_{21}\in E\\\text{such that }f_{11}f_{22}\geq f_{12}^2+f_{21}^2\end{gathered}\right\}.
\end{equation}
Using Definition~\ref{D147}, \eqref{H>}, and \eqref{E+}, we  conclude that
\[\begin{split}
    \cN_1(\Lambda_L)
    &=\ker( \Lambda_L)\cap\cH_{2,\succeq}(E)
    =\{F\in\cH_{2,\succeq}(E)\colon f_{11}(1)=0\text{ and }f_{22}(-1)=0\}\\
    &=\left\{F\in\cH_{2,\succeq}(E)\colon F(1)=\left(\begin{smallmatrix}0&0\\0&f_{22}(1)\end{smallmatrix}\right)\text{ and }F(-1)=\left(\begin{smallmatrix}f_{11}(-1)&0\\0&0\end{smallmatrix}\right)\right\}\\
    &=\left\{\begin{pmatrix}a(x-1)^2&\overline{w}(x-1)(x+1)\\w(x-1)(x+1)&d(x+1)^2\end{pmatrix}\colon\begin{gathered} a,d\in[0,\infty)\text{ and }w\in\dC\\\text{such that }ad\geq\lvert w\rvert^2\end{gathered}\right \}
\end{split}\]
and thus 
\[\begin{split}
    \cV_1(\Lambda_L)
    &=\bigcap_{F\in\cN_1(\Lambda_L)}\{\det F=0\}
    =\bigcap_{\substack{a,d\in[0,\infty),\;w\in\dC\\\text{with }ad\geq\lvert w\rvert^2 }}\{(ad-\lvert w\rvert^2)(x-1)^2(x+1)^2=0\}\\
    &=\{-1,1\}=\cW(L)=\cW(\Lambda_L)
\end{split}\]
by \eqref{equWL}.
Since $\cV_j(\Lambda_L)\subseteq \cV_1(\Lambda_L)$ as stated in Lemma~\ref{P148-aP}\ref{P148-aP.b}, $\cV(\Lambda_L)\subseteq \{1,-1\}$.
Thus, the core set $\cV(\tr L)$ is different from the core set $\cV(\Lambda_L)$.
In fact, we have $\cV(\Lambda_L)=\{1,-1\}=\cW(\Lambda_L)$ by Theorem~\ref{T149} below.
\end{exm}

Our next main aim is  Theorem~\ref{P148-c} below.
First we develop a number of preliminary results.

\begin{lem}\label{P148-aP}
\begin{enumerate}[(a)]
    \item\label{P148-aP.a} $\cN_k(\Lambda)\subseteq\cN_{k+1}(\Lambda)$ for  $k\in\dN$.
    \item\label{P148-aP.b} $\cV_{j+1}(\Lambda)\subseteq\cV_j(\Lambda)$  and $\bV_{j+1}(\Lambda;x)\subseteq\bV_j(\Lambda;x)$ for  $x\in\cV_j(\Lambda)$ and  $j\in\dN_0$.
\end{enumerate}
\end{lem}
\begin{proof}
\ref{P148-aP.b}:
Let $j\in\dN_0$.
By construction,  $\cV_{j+1}(\Lambda)\subseteq\cV_j(\Lambda)$.
Let $x\in\cV_j(\Lambda)$ and $v\in\bV_{j+1}(\Lambda;x)$.
From \eqref{VVj} we  infer $v\in\ker P_j^\bot(x)$ which implies $v\in\bV_{j}(\Lambda;x)$.

\ref{P148-aP.a}:
Recall that $\cV_0(\Lambda)=\cX$ and $\bV_0(\Lambda;x)=\dC^q$.
Hence $P_0(x)=I_q$ for  $x\in\cX$.
Let $F\in\cN_{1}(\Lambda)$.
Then $F\in\ker\Lambda$ and $F(x)=(P_{0}FP_{0})(x)\succeq  O$ for $x\in\cV_{0}(\Lambda)=\cX$.
In particular, for $x\in\cV_1(\Lambda)$, we have $F(x)\succeq  O$ and hence $(P_1FP_1)(x)\succeq  O$.
Thus, $F\in\cN_{2}(\Lambda)$.
This proves that $\cN_{1}(\Lambda)\subseteq\cN_{2}(\Lambda)$.

We proceed by induction. Suppose  $\cN_{\ell}(\Lambda)\subseteq\cN_{\ell+1}(\Lambda)$ is already proved for some $\ell\in\dN$.
Let $F\in\cN_{\ell+1}(\Lambda)$.
Then $F\in\ker\Lambda$ and $(P_{\ell}FP_{\ell})(x)\succeq  O$ for all $x\in\cV_{\ell}(\Lambda)$.
From $\cV_{\ell+1}(\Lambda)\subseteq\cV_{\ell}(\Lambda)$ we obtain $(P_{\ell}FP_{\ell})(x)\succeq  O$ for all $x\in\cV_{\ell+1}(\Lambda)$.
Since $\bV_{\ell+1}(\Lambda;x)\subseteq\bV_{\ell}(\Lambda;x)$ for $x\in\cV_{\ell+1}(\Lambda)$ by~\ref{P148-aP.b}, we have $P_{\ell+1}P_\ell=P_\ell P_{\ell+1}=P_{\ell+1}$ which implies $P_{\ell+1}(P_\ell FP_\ell)P_{\ell+1}=P_{\ell+1}FP_{\ell+1}$.
Therefore, $(P_{\ell+1}FP_{\ell+1})(x)\succeq  O$ for all $x\in\cV_{\ell+1}(\Lambda)$.
Thus, $F\in\cN_{\ell+2}(\Lambda)$.
This proves that $\cN_{\ell+1}(\Lambda)\subseteq\cN_{\ell+2}(\Lambda)$.
\end{proof}

\begin{lem}\label{L1841P}
Let $\mu\in\cM_\Lambda$,  $j\in\dN_0$, and  $\cX_j:= \cV_j(\Lambda)$.
Suppose  $\cX_j\in\fX\setminus\{\emptyset\}$,  $\mu(\cX\setminus\cX_j)=O$, and $\ran (\mu(\{x\}))\subseteq\bV_j(\Lambda;x)$ for all $x\in\cX_j$.
Further, suppose that $P_j\in\cM(\cX_j,\fX_j;\cH_q)$, where $\fX_j:=\cX_j\cap\fX$.
Then:
\begin{enumerate}[(a)]
    \item\label{L1841P.a} $\mu_j:= \mu\lceil{\fX_j}$ belongs to $\cM_q(\cX_j,\fX_j)$ and satisfies $\ats (\mu_j)=\ats (\mu)$.
    \item\label{L1841P.b} Set  $F_j := F\lceil{\cX_j}$ for $F\in\cE$.
    Then $ \cE_j := \{P_jF_jP_j\colon F\in\cE\}$ is a linear subspace of $L^1(\mu_j;\cH_q)$ with $\dim\cE_j \leq\dim\cE$ and $\Lambda(F)=\Lambda^{\mu_j}(P_jF_jP_j)$ for  $F\in\cE$.
    \item\label{L1841P.c} Let $ \Lambda_j:=\Lambda^{\mu_j}$.
    Then $ \Lambda_j \in\fL(\cE_j)$ and $\mu_j\in\cM_{\Lambda_j}$.
    \item\label{L1841P.d} $\cN_+(\Lambda_j,P_j)=\{P_jF_jP_j\colon F\in\cN_{j+1}(\Lambda)\}$,~  $\cV_+(\Lambda_j,P_j)=\cV_{j+1}(\Lambda)$, and $\bV_+(\Lambda_j,P_j;x)=\bV_{j+1}(\Lambda;x)$ for all $x\in\cX_j$.
\end{enumerate}
\end{lem}
\begin{proof}
First we observe that the assumptions imply that $\fX_j$ is a sub-$\sigma$-algebra of $\fX$ on $\cX_j$, so that $(\cX_j,\fX_j)$ is a non-trivial measurable space.

\ref{L1841P.a}:
Clearly, $\mu_j$ is a well-defined measure of $\cM_q(\cX_j,\fX_j)$.
By $\mu(\cX\setminus\cX_j)=O$, we have $\ats (\mu_j)=\ats (\mu)$.

\ref{L1841P.b}:
Consider an arbitrary $F\in\cE$.
Then $F\in L^1(\mu;\cH_q)$.
Let $\rnd $ be  the Radon--Nikodym derivative of $\mu$ with respect to its trace measure $\tau$.
By ~\ref{L1841P.a},  $\tau_j:=\tau\lceil{\fX_j}$ is the trace measure of $\mu_j$ and $\rnd_j:=\rnd \lceil{\cX_j}$ is  the Radon--Nikodym derivative of $\mu_j$ with respect to $\tau_j$.
Because of $\ran (\mu(\{x\}))\subseteq\bV_j(\Lambda;x)$ for $x\in\cX_j$, we have $P_j\rnd_jP_j=\rnd_j$.
Since $\mu\in\cM_\Lambda$ and $\mu(\cX\setminus\cX_j)=O$, then
\[\begin{split}
    \Lambda(F)
    &=\Lambda^\mu(F)
    =\int_{\cX }\langle F,\rnd \rangle\dif\tau
    =\int_{\cX_j}\langle F,\rnd \rangle\dif\tau
    =\int_{\cX_j}\langle F_j,\rnd_j\rangle\dif\tau_j\\
    &=\int_{\cX_j}\langle F_j,P_j\rnd_jP_j\rangle\dif\tau_j
    =\int_{\cX_j}\langle P_jF_jP_j,\rnd_j\rangle\dif\tau_j
    =\Lambda^{\mu_j}(P_jF_jP_j).
\end{split}\]
In particular, $P_jF_jP_j\in L^1(\mu_j;\cH_q)$.
Hence $\cE_j$ is a linear subspace of $L^1(\mu_j;\cH_q)$.
We have $\dim \cE_j\leq \dim\cE$, since each basis of $\cE$ gives rise to a spanning set of $\cE_j$.

\ref{L1841P.c} is obvious from~\ref{L1841P.a} and~\ref{L1841P.b}.

\ref{L1841P.d}:
Let $G\in\cN_+(\Lambda_j,P_j)$.
Then $G\in\ker\Lambda_j$ and $(P_jGP_j)(x)\succeq O$ for all $x\in\cX_j$.
According to ~\ref{L1841P.b}, there exists $F\in\cE$ with $G=P_jF_jP_j$ and furthermore $\Lambda(F)= \Lambda_j(P_jF_jP_j)= \Lambda_j(G)=0$.
Consequently, $P_jGP_j=P_j^2F_jP_j^2=P_jF_jP_j$ and $F\in\ker\Lambda$.
For  $x\in\cX_j$, we have $(P_jFP_j)(x)=(P_jF_jP_j)(x)=(P_jGP_j)(x)$.
In particular, $(P_jFP_j)(x)\succeq O$ for  $x\in\cX_j=\cV_j(\Lambda)$.
Therefore, $F\in\cN_{j+1}(\Lambda)$.
Now consider let $F\in\cN_{j+1}(\Lambda)$.
Then $F\in\ker\Lambda$ and $(P_jFP_j)(x)\succeq O$ for all $x\in\cV_j(\Lambda)$.
In particular, $F\in\cE$ and $\Lambda(F)=0$.
By \ref{L1841P.b}, $G:=P_jF_jP_j$ belongs to $\cE_j $ and fulfills $\Lambda_j(G)=\Lambda_j(P_jF_jP_j)= \Lambda(F)=0$.
Hence $G\in\ker\Lambda_j$.
Since $(P_jFP_j)(x)\succeq O$ for $x\in\cV_j(\Lambda)$, we have $(P_jGP_j)(x)=(P_j^2F_jP_j^2)(x)=(P_jF_jP_j)(x)=(P_jFP_j)(x)\succeq   O $ for $x\in\cX_j$.
Thus, $G\in\cN_+(\Lambda_j,P_j)$.
We have shown that 
\begin{equation}\label{L1841P.1}
    \cN_+(\Lambda_j,P_j)
    =\{P_jF_jP_j\colon F\in\cN_{j+1}(\Lambda)\}.
\end{equation}
Taking  into account \eqref{VggP}, \eqref{L1841P.1}, and \eqref{1.23}, we derive
\[\begin{split}
    \cV_+(\Lambda_j,P_j)
    &=\bigcap_{G\in\cN_+(\Lambda_j,P_j)}\{x\in\cX_j\colon\det(P_jGP_j+P_j^\bot)(x)=0\}\\
    &=\bigcap_{F\in\cN_{j+1}(\Lambda)}\{x\in\cX_j\colon\det(P_j^2F_jP_j^2+P_j^\bot)(x)=0\}\\
    &=\bigcap_{F\in\cN_{j+1}(\Lambda)}\{x\in\cX_j\colon\det(P_jF_jP_j+P_j^\bot)(x)=0\}\\
    &=\bigcap_{F\in\cN_{j+1}(\Lambda)}\{x\in\cX_j\colon\det(P_jFP_j+P_j^\bot)(x)=0\}
    =\cV_{j+1}(\Lambda).
\end{split}\]

Now let $x\in\cX_j$.
Using  \eqref{VV+P}, \eqref{L1841P.1}, and \eqref{VVj}, we deduce
\[\begin{split}
    \bV_+(\Lambda_j,P_j;x)
    &=\bigcap_{G\in\cN_+(\Lambda_j,P_j)}\ker(P_jGP_j+P_j^\bot)(x)\\
    &=\bigcap_{F\in\cN_{j+1}(\Lambda)}\ker(P_j^2F_jP_j^2+P_j^\bot)(x)\\
    &=\bigcap_{F\in\cN_{j+1}(\Lambda)}\ker(P_jF_jP_j+P_j^\bot)(x)\\
    &=\bigcap_{F\in\cN_{j+1}(\Lambda)}\ker(P_jFP_j+P_j^\bot)(x)
    =\bV_{j+1}(\Lambda;x).\qedhere
\end{split}\]
\end{proof}

Let $M_{p,q}(\dC)$ denote the complex $p\times q$-matrices.
For each matrix $A\in M_{p,q}(\dC)$, there exists a unique matrix $X\in M_{q,p}(\dC)$, satisfying the equations (see,  e.\,g., \cite[Ch.~1]{MR1987382})
\begin{align*}%
    AXA&=A,&
    XAX&=X&
    &\text{and}&
    (AX)^*&=AX,&
    (XA)^*&=XA.
\end{align*}
This matrix $X$ is called the \emph{Moore--Penrose inverse of $A$} and  denoted by $A^\dagger$.

\begin{rem}[see, e.\,g., {\cite[Ex.~2.58, p.~80]{MR1987382}}]\label{L2022}
If $A\in M_{p,q}(\dC)$, then $A^\dagger A=\dP_{\ran(A^*)}$.
\end{rem}

\begin{rem}[see, e.\,g., {\cite[Ex.~3.25, p.~115]{MR1987382}}]\label{L1336}
If $A\in M_{p,q}(\dC)$, then $A^\dagger=\lim_{n\to\infty}(A^*A+\frac{1}{n}I_q)^{-1}A^*$.
\end{rem}

\begin{lem}\label{L1228}
The mapping $\Pi\colon\cH_q\to\cH_q$ defined by $\Pi(A):=\dP_{\ker A}$ is measurable.
\end{lem}
\begin{proof}
The mapping $\Pi_n\colon\cH_q\to\cH_q$ defined by $\Pi_n(A):=I_q-(A^*A+\frac{1}{n}I_q)^{-1}A^*A$ is measurable for each $n\in \dN$.
Using Remarks~\ref{L2022} and~\ref{L1336},  for all $A\in\cH_q$ we infer
\[
    \Pi(A)
    =\dP_{[\ran(A^*)]^\bot}
    =I_q-A^\dagger A
    =I_q-[\lim_{n\to\infty}(A^*A+\frac{1}{n}I_q)^{-1}A^*]A
    =\lim_{n\to\infty}\Pi_n(A),
\]
so  $\Pi$ is a pointwise limit of measurable mappings.
Hence $\Pi$ is measurable as well.
\end{proof}

\begin{lem}\label{L1402}
Let $F\in\cM(\cX,\mathfrak{X};\cH_q)$.
Then $P\colon\cX\to\cH_q$ defined by $P(x):=\dP_{\ker F(x)}$ belongs to $\cM(\cX,\mathfrak{X};\cH_q)$.
\end{lem}
\begin{proof}
The assertion follows from Lemma~\ref{L1228}, since $P=\Pi\circ F$.
\end{proof}

\begin{lem}\label{L1925}
Let $\mu\in\cM^\mathrm{fa}_\Lambda$.
For $j\in\dN_0$, set $\cX_j:=\cV_j(\Lambda)$ and  $\fX_j:=\cX_j\cap\fX$.
Then, for  $j\in\dN_0$, we have:
\begin{enumerate}[(a)]
    \item\label{L1925.a} $\ats (\mu)\subseteq\cX_j$,
    \item\label{L1925.b} $\ran (\mu(\{x\}))\subseteq\bV_j(\Lambda;x)$ for all $x\in\cX_j$,
    \item\label{L1925.c} $\cX_j\in\fX\setminus\{\emptyset\}$ and $\fX_j$ is a sub-$\sigma$-algebra of $\fX$ on $\cX_j$,
    \item\label{L1925.d} $(\cX_j,\fX_j)$ is a non-trivial measurable space and $P_j\in\cM(\cX_j,\fX_j;\cH_q)$.
\end{enumerate}
\end{lem}
\begin{proof}
The proof  proceeds by  induction on $j\in \dN_0$.

Clearly, \ref{L1925.a}--\ref{L1925.d} hold for $j=0$, since $\bV_0(\Lambda;x)=\dC^q$ and $P_0(x)=I_q$ for $x\in \cX_0=\cX$.

Now assume that~\ref{L1925.a}--\ref{L1925.d} hold for some $j\in\dN_0$.
Then $\mu(\cX\setminus\cX_j)=O$ by $\mu\in\cM^\mathrm{fa}_q(\cX,\fX)$ and~\ref{L1925.a}.
Lemma~\ref{L1841P}\ref{L1841P.a} applies and shows that $\mu_j:= \mu\lceil{\fX_j}$ belongs to $\cM_q(\cX_j,\fX_j)$ and  $\ats (\mu_j)=\ats (\mu)$.
Hence $\mu_j\in\cM^\mathrm{fa}_q(\cX_j,\fX_j)$.
Since $\mu_j(\{x\})=\mu(\{x\})$ for $x\in\cX_j$, we get  $\ran (\mu_j(\{x\}))\subseteq\bV_j(\Lambda;x)$ by~\ref{L1925.b}.
For $F\in\cE$, set $F_j = F\lceil{\cX_j}$.
By Lemma~\ref{L1841P}\ref{L1841P.b}, $ \cE_j := \{P_jF_jP_j\colon F\in\cE\}$ is a linear subspace of $L^1(\mu_j;\cH_q)$ with $\dim\cE_j \leq\dim\cE$ and $\Lambda(F)=\Lambda^{\mu_j}(P_jF_jP_j)$ for all $F\in\cE$.
In particular, $\cE_j$ is finite-dimensional.
Let $ \Lambda_j:=\Lambda^{\mu_j}$.
According to Lemma~\ref{L1841P}\ref{L1841P.c}, $ \Lambda_j \in\fL(\cE_j)$ and $\mu_j\in\cM_{\Lambda_j}$.
Clearly, $\mu_j\in\cM^\mathrm{fa}_{\Lambda_j}$, since $\mu\in\cM^\mathrm{fa}_{\Lambda}$.
Now we apply Proposition~\ref{P140iP} to the measurable space $(\cX_j,\fX_j)$, the subspace $\cE_j$ of $L^1(\mu_j;\cH_q)$, the functional $\Lambda_j\in\fL(\cE_j)$, its representing measure $\mu_j\in\cM_{\Lambda_j}$, and the family $(\bV_j(\Lambda;x))_{x\in\cX_j}$ of subspaces of $\dC^q$.
This yields
\begin{equation}\label{L1925.7}
    \ats (\mu_j)\subseteq\cV_+(\Lambda_j,P_j)
    \text{ and }
    \ran (\mu_j(\{x\}))\subseteq\bV_+(\Lambda_j,P_j;x)\text{ for } x\in\cX_j.
\end{equation}
Lemma~\ref{L1841P}\ref{L1841P.d} provides $\cN_+(\Lambda_j,P_j)=\{P_jF_jP_j\colon F\in\cN_{j+1}(\Lambda)\}$ as well as 
\begin{equation}\label{L1925.2}
    \cV_+(\Lambda_j,P_j)=\cX_{j+1}
    \text{ and }
    \bV_+(\Lambda_j,P_j;x)=\bV_{j+1}(\Lambda;x)\text{ for all }x\in\cX_{j}.
\end{equation}
Summarizing, we get
\begin{equation}\label{L1925.1}
    \ats (\mu)
    =\ats (\mu_j)
    \subseteq\cV_+(\Lambda_j,P_j)
    =\cX_{j+1},
\end{equation} 
i.\,e., we have proved~\ref{L1925.a} for $j+1$.

By $\mu\in\cM^\mathrm{fa}_q(\cX,\fX)$ and \eqref{L1925.1},  $\mu(\cX\setminus\cX_{j+1})=O$.
Since $\mu(\cX)\neq O$  by the assumption $\Lambda\neq 0$ stated at the beginning of this section, this implies $\cX_{j+1}\neq\emptyset$.
Now let $\xi\in\cX_{j+1}$.
Then $\xi\in \cX_j$ by Lemma~\ref{P148-aP}\ref{L1925.b}.
Hence, $\mu(\{\xi\})=\mu_j(\{\xi\})$.
Taking \eqref{L1925.7} and \eqref{L1925.2} into account, we  conclude that
\[\begin{split}
    \ran (\mu(\{\xi\}))
    =\ran (\mu_j(\{\xi\}))
    \subseteq\bV_+(\Lambda_j,P_j;\xi)
    =\bV_{j+1}(\Lambda;\xi).
\end{split}\]
Hence,~\ref{L1925.b} holds true for $j+1$.

From Proposition~\ref{P144P}, applied to the measurable space $(\cX_j,\fX_j)$, the  finite-dimensional  subspace $\cE_j$ of $L^1(\mu_j;\cH_q)$, the functional $\Lambda_j\in\fL(\cE_j)$, and the mapping $P_j\colon\cX_j\to\cH_q$, we conclude that there exists $S^{(j)}\in\cN_+(\Lambda_j,P_j)$ such that $\cV_+(\Lambda_j,P_j)=\{x\in\cX_j\colon\det(P_jS^{(j)}P_j+P_j^\bot)(x)=0\}$ and $\bV_+(\Lambda_j,P_j;x)=\ker(P_jS^{(j)}P_j+P_j^\bot)(x)$ for all $x\in\cX_j$.
Then there exists $F^{(j+1)}\in\cN_{j+1}(\Lambda)$ such that $S^{(j)}=P_jF^{(j+1)}_jP_j$.
By \eqref{L1925.2}, 
\begin{equation}\label{L1925.5}\begin{split}
    \cX_{j+1}
    =\cV_+(\Lambda_j,P_j)
    &=\{x\in\cX_j\colon\det(P_jS^{(j)}P_j+P_j^\bot)(x)=0\}\\
    &=\{x\in\cX_j\colon\det(P_j^2F^{(j+1)}_jP_j^2+P_j^\bot)(x)=0\}\\
    &=\{x\in\cX_j\colon\det(P_jF^{(j+1)}_jP_j+P_j^\bot)(x)=0\}
\end{split}\end{equation}
and
\begin{equation}\label{L1925.6}\begin{split}
    \bV_{j+1}(\Lambda;x)
    &=\bV_+(\Lambda_j,P_j;x)
    =\ker(P_jS^{(j)}P_j+P_j^\bot)(x)\\
    &=\ker(P_j^2F^{(j+1)}_jP_j^2+P_j^\bot)(x)\\
    &=\ker(P_jF^{(j+1)}_jP_j+P_j^\bot)(x)\qquad\text{for all }x\in\cX_{j}.
\end{split}\end{equation}
Because of \eqref{1.22}, we have $F^{(j+1)}\in\ker\Lambda\subseteq\cE\subseteq\cM(\cX,\mathfrak{X};\cH_q)$.
Using~\ref{L1925.c} we  infer $F^{(j+1)}_j\in\cM(\cX_j,\fX_j;\cH_q)$.
Since~\ref{L1925.d}  implies $P_j,P_j^\bot\in\cM(\cX_j,\fX_j;\cH_q)$, then $G:=P_jF^{(j+1)}_jP_j+P_j^\bot$ belongs to $\cM(\cX_j,\fX_j;\cH_q)$, so that $\det G\in\cM(\cX_j,\fX_j;\dR)$.
From \eqref{L1925.5} we obtain $\cX_{j+1}\in\fX_j$.
In view of $\fX_j\subseteq\fX$ by~\ref{L1925.c} and $\cX_{j+1}\neq\emptyset$, then $\cX_{j+1}\in\fX\setminus\{\emptyset\}$.
Thus,~\ref{L1925.c} holds  for $j+1$.

Set $H:=G\lceil{\cX_{j+1}}$.
We  show that $H\in\cM(\cX_{j+1},\fX_{j+1};\cH_q)$.
For let $B\in\mathfrak{B}(\cH_q)$.
In view of $G\in\cM(\cX_j,\fX_j;\cH_q)$, then $G^{-1}(B)\in\fX_j$.
From $\fX_j\subseteq\fX$ we have $G^{-1}(B)\in\fX$.
Because of $H^{-1}(B)=\cX_{j+1}\cap G^{-1}(B)$, then $H^{-1}(B)\in\cX_{j+1}\cap\mathfrak{X}=\fX_{j+1}$.
Thus, we have proved $H\in\cM(\cX_{j+1},\fX_{j+1};\cH_q)$.
Recall that $\cX_{j+1}\subseteq\cX_{j}$ by Lemma~\ref{P148-aP}\ref{P148-aP.b}.
Using \eqref{L1925.6} for $x\in\cX_{j+1}$ we get
\[
    P_{j+1}(x)
    =\dP_{\bV_{j+1}(\Lambda;x)}
    =\dP_{\ker(P_jF^{(j+1)}_jP_j+P_j^\bot)(x)}
    =\dP_{\ker G(x)}
    =\dP_{\ker H(x)}.
\]
Hence we can apply Lemma~\ref{L1402} to obtain $P_{j+1}\in\cM(\cX_{j+1},\fX_{j+1};\cH_q)$.
Thus,~\ref{L1925.d} holds true for $j+1$.
This completes the induction proof.
\end{proof}

\begin{prop}\label{P148-bP}
Let $\mu\in\cM^\mathrm{fa}_\Lambda$.
Then $\ats (\mu)\subseteq\cV( \Lambda)$ and $\ran (\mu(\{x\}))\subseteq\bV( \Lambda;x)$ for all  $x\in\cV(\Lambda)$.
\end{prop}
\begin{proof}
Lemma~\ref{L1925} yields $\ats (\mu)\subseteq\cV_j(\Lambda)$ as well as $\ran (\mu(\{x\}))\subseteq\bV_j(\Lambda;x)$ for all $x\in\cV_j(\Lambda)$ and $j\in \dN_0$.
By \eqref{1.24}, then $\ats (\mu)\subseteq\bigcap_{j=0}^\infty\cV_j(\Lambda)=\cV( \Lambda)$ and $\ran (\mu(\{x\}))\subseteq\bigcap_{j=0}^\infty\bV_{j}(\Lambda;x)=\bV( \Lambda;x)$ for all $x\in\bigcap_{j=0}^\infty\cV_j(\Lambda)=\cV( \Lambda)$.
\end{proof}

\begin{lem}\label{L0920}
$\cV( \Lambda)\in\fX\setminus\{\emptyset\}$.
\end{lem}
\begin{proof}
According to Theorem~\ref{richter}, there exists $\mu\in\cM^\mathrm{fa}_\Lambda$.
Thus, we can apply Lemma~\ref{L1925} to obtain $\cV_j(\Lambda)\in\fX$ for all $j\in\dN_0$.
By \eqref{1.24},  $\cV( \Lambda)\in\fX$.
Because of $\mu\in\cM^\mathrm{fa}_q(\cX,\fX)$ and $\mu(\cX)\neq O$ by the assumption $\Lambda\neq 0$, we have $\ats (\mu)\neq\emptyset$.
Proposition~\ref{P148-bP} yields $\ats (\mu)\subseteq\cV( \Lambda)$.
Consequently, $\cV( \Lambda)\neq\emptyset$.
\end{proof}

\begin{thm}\label{P148-c}
Suppose  $\Lambda$ is a moment functional on $\cE$ such that $\Lambda\neq 0$.
Then the sets $\cV_j(\Lambda)$, $j\in \dN_0$, and $\cV(\Lambda)$ are measurable and there exists  $k\in\dN$ such that
\begin{equation}\label{1.25}
    \cX 
    =\cV_0(\Lambda)
    \supseteq\dotsb
    \supseteq\cV_k(\Lambda)
    =\cV_{k+\ell}(\Lambda)
    =\cV( \Lambda)\text{ for all }\ell\in\dN
\end{equation}
and
\begin{multline}\label{1.25+}
    \dC^q
    =\bV_0(\Lambda;\xi)
    \supseteq\dotsb
    \supseteq\bV_{k}(\Lambda;\xi)
    =\bV_{k+\ell}(\Lambda;\xi)
    =\bV( \Lambda;\xi)
    \neq \{0\}\\
    \text{for all }\xi\in\cV( \Lambda)\text{ and }\ell\in\dN.
\end{multline}
\end{thm}
\begin{proof}
    By Lemma~\ref{L1925}\ref{L1925.c} each set $\cV_j(\Lambda)$ is measurable and so is $\cV(\Lambda)$ by \eqref{1.24}.

By Lemma~\ref{P148-aP}  and \eqref{1.24}, we have
\begin{align}
    \cX 
    =\cV_0(\Lambda)
    \supseteq\cV_{1}(\Lambda)
    &\supseteq\cV_{2}(\Lambda)
    \supseteq\dotsb,\label{P148-c.1}\\
    \dC^q
    =\bV_0(\Lambda;x)
    \supseteq\bV_{1}(\Lambda;x)
    &\supseteq\bV_{2}(\Lambda;x)
    \supseteq\dotsb&\text{for all }x&\in\cV(\Lambda).\label{P148-c.2}
\end{align}
According to  Theorem~\ref{richter}, there exists $\mu\in\cM^\mathrm{fa}_\Lambda$. 
Let $j\in\dN_0$.
From Lemma~\ref{L1925} we  infer that $\ats (\mu)\subseteq\cX_j$ and $\ran (\mu(\{x\}))\subseteq\bV_j(\Lambda;x)$ for all $x\in\cX_j$ as well as $\cX_j\in\fX\setminus\{\emptyset\}$ and $P_j\in\cM(\cX_j,\fX_j;\cH_q)$, where  $\cX_j$ and $\fX_j$ are as in Lemma~\ref{L1925}.
Since $\mu\in\cM^\mathrm{fa}_q(\cX,\fX)$,  in particular $\mu(\cX\setminus\cX_j)=O$.
Consequently, we can apply Lemma~\ref{L1841P}.
By Lemma~\ref{L1841P}\ref{L1841P.a}, $\mu_j:= \mu\lceil{\fX_j}$ belongs to $\cM_q(\cX_j,\fX_j)$.
Clearly, $\mu_j\in\cM^\mathrm{fa}_q(\cX_j,\fX_j)$.
For $F\in\cE$, let $F_j := F\lceil{\cX_j}$.
Lemma~\ref{L1841P}\ref{L1841P.b} then shows that $ \cE_j := \{P_jF_jP_j\colon F\in\cE\}$ is a linear subspace of $L^1(\mu_j;\cH_q)$ with $\dim\cE_j \leq\dim\cE$. 
In particular, $\cE_j$ is finite-dimensional.
Let $ \Lambda_j:=\Lambda^{\mu_j}$.
By Lemma~\ref{L1841P}\ref{L1841P.c}, $ \Lambda_j:=\Lambda^{\mu_j}$ belongs to $\fL(\cE_j)$.
Lemma~\ref{L1841P}\ref{L1841P.d} implies $\cN_+(\Lambda_j,P_j)=\{P_jF_jP_j\colon F\in\cN_{j+1}(\Lambda)\}$ and that \eqref{L1925.2} holds.

From Proposition~\ref{P144P}, applied to the measurable space $(\cX_j,\fX_j)$, the  finite-dimensional  subspace $\cE_j$ of $L^1(\mu_j;\cH_q)$, the functional $\Lambda_j\in\fL(\cE_j)$, and the projector-valued mapping $P_j$, it follows that there is $S^{(j)}\in\cN_+(\Lambda_j,P_j)$ such that
\begin{align*}
    \cV_+(\Lambda_j,P_j)&=\{\det(P_jS^{(j)}P_j+P_j^\bot)=0\},\\
    \bV_+(\Lambda_j,P_j;x)&=\ker(P_jS^{(j)}P_j+P_j^\bot)(x)
    \quad\text{for all }x\in\cX_j.
\end{align*}
Then there exists $F^{(j+1)}\in\cN_{j+1}(\Lambda)$ such that $S^{(j)}=P_jF^{(j+1)}_jP_j$.
In view of \eqref{L1925.2}, we get  \eqref{L1925.5} and \eqref{L1925.6}.
Taking into account \eqref{P148-c.1}, we can infer from \eqref{L1925.5} by induction that
\begin{equation}\label{P148-c.3}
    \cX_{s}
    =\bigcap_{r=1}^s\{x\in\cX_{s-1}\colon\det(P_{r-1}F^{(r)}P_{r-1}+P_{r-1}^\bot)(x)=0\}\quad\text{for }s\in\dN.
\end{equation}
According to \eqref{L1925.6}, we have
\begin{equation}\label{P148-c.4}
    \bV_{s}(\Lambda;x)
    =\ker(P_{s-1}F^{(s)}P_{s-1}+P_{s-1}^\bot)(x)
    \quad\text{for all }x\in\cX_{s-1},\; s\in\dN.
\end{equation}

Because of $\dim\cE<\infty$ and $F^{(s)}\in\cE$, there exists $k\in\dN$ such that
\begin{equation}\label{P148-c.8}
    F^{(s)}
    =\sum\nolimits_{r=1}^k\eta^{(s)}_{r}F^{(r)},
\end{equation}
where $\eta^{(s)}_{1},\dotsc,\eta^{(s)}_{k}\in\dR$ for $s\in\dN$.
We now show by induction that
\begin{equation}\label{P148-c.9}
    \cX_{k}=\cX_{k+\ell}~
    \text{ and }~
    \bV_{k}(\Lambda;x)
    =\bV_{k+\ell}(\Lambda;x)~ \text{ for all }~ x\in\cX_{k} ~
    \text{and }\ell\in\dN.
\end{equation}
Let $s\in\dN$, $s\geq k+1$, and $x\in\cX_{s-1}$.
For  $r\in\{1,\dotsc,k\}$, it follows from Lemma~\ref{P148-aP}\ref{P148-aP.b} that $x\in\cX_{r-1}$ and $P_{s-1}(x)P_{r-1}(x)=P_{r-1}(x)P_{s-1}(x)=P_{s-1}(x)$, so that $(P_{s-1}F^{(r)}P_{s-1})(x)=P_{s-1}(x)(P_{r-1}F^{(r)}P_{r-1})(x)P_{s-1}(x)$.
By \eqref{P148-c.8} we get
\begin{equation}\label{P148-c.5}\begin{split}
    (P_{s-1}F^{(s)}P_{s-1})(x)
    & =P_{s-1}(x)\left(\sum_{r=1}^k\eta^{(s)}_{r}F^{(r)}(x)\right)P_{s-1}(x)\\
    & =\sum_{r=1}^k\eta^{(s)}_{r}P_{s-1}(x)(P_{r-1}F^{(r)}P_{r-1})(x)P_{s-1}(x)
\end{split}\end{equation}
for all $x\in\cX_{s-1}$ and $s\in\{k+1,k+2,\dotsc\}$.
Let $\xi\in\cX_k$ and  $v\in\bV_k(\Lambda;\xi)$.
Then $[P_k(\xi)]v=v$ and $[P_k^\bot(\xi)]v=0$.
By  Lemma~\ref{P148-aP}\ref{P148-aP.b}, for  $r\in\{1,\dotsc,k\}$, we have $\xi\in\cX_{r-1}$ and $v\in\bV_r(\Lambda;\xi)$.
In particular, $\xi\in\cX_{k-1}$.
For  $r\in\{1,\dotsc,k\}$, by \eqref{P148-c.4}, then $v\in\ker(P_{r-1}F^{(r)}P_{r-1}+P_{r-1}^\bot)(\xi)$, so that $[(P_{r-1}F^{(r)}P_{r-1})(\xi)]v=-[P_{r-1}^\bot(\xi)]v$.
This implies $[(P_{r-1}F^{(r)}P_{r-1})(\xi)]v=0$ and $[P_{r-1}^\bot(\xi)]v=0$.
Taking into account \eqref{P148-c.5} for $s=k+1$, we obtain
\[\begin{split}
    [(P_{k}F^{(k+1)}P_{k}+P_{k}^\bot)(\xi)]v
    &=[(P_{k}F^{(k+1)}P_{k})(\xi)]v\\
    &=\sum_{r=1}^k\eta^{(k+1)}_{r}[P_{k}(\xi)][(P_{r-1}F^{(r)}P_{r-1})(\xi)][P_{k}(\xi)]v\\
    &=\sum_{r=1}^k\eta^{(k+1)}_{r}[P_{k}(\xi)][(P_{r-1}F^{(r)}P_{r-1})(\xi)]v
    =0,
\end{split}\]
i.\,e., $v\in\ker(P_{k}F^{(k+1)}P_{k}+P_{k}^\bot)(\xi)$.
By $\xi\in\cX_k$ and \eqref{P148-c.4}, then $v\in\bV_{k+1}(\Lambda;\xi)$.
Since $v\in\bV_{k}(\Lambda;\xi)$ was arbitrary, we have shown $\bV_{k}(\Lambda;\xi)\subseteq\bV_{k+1}(\Lambda;\xi)$.
Using  Lemma~\ref{P148-aP}\ref{P148-aP.b}  it follows that $\bV_{k}(\Lambda;\xi)=\bV_{k+1}(\Lambda;\xi)$.

In view of $\xi\in\cX_k$ and \eqref{P148-c.3}, we have $\det(P_{r-1}F^{(r)}P_{r-1}+P_{r-1}^\bot)(\xi)=0$ for  $r\in\{1,\dotsc,k\}$.
In particular, $\det(P_{k-1}F^{(k)}P_{k-1}+P_{k-1}^\bot)(\xi)=0$.
By $\xi\in\cX_{k-1}$ and \eqref{P148-c.4}, we obtain $\bV_{k}(\Lambda;\xi)\neq\{0\}$.
Hence, $\bV_{k+1}(\Lambda;\xi)\neq\{0\}$.
Because of $\xi\in\cX_k$ and \eqref{P148-c.4}, then $\det(P_{k}F^{(k+1)}P_{k}+P_{k}^\bot)(\xi)=0$.
Hence, $\det(P_{r-1}F^{(r)}P_{r-1}+P_{r-1}^\bot)(\xi)=0$ for $r\in\{1,\dotsc,k+1\}$.
Then $\xi\in\cX_{k+1}$ by \eqref{P148-c.3}.
Since $\xi\in\cX_k$ was arbitrary, we have shown $\cX_k\subseteq\cX_{k+1}$.
Hence  $\cX_k=\cX_{k+1}$ by \eqref{P148-c.1}.

Now we begin the induction proof of \eqref{P148-c.9}.
Assume that $\cX_k=\cX_{k+\ell}$ and $\bV_{k}(\Lambda;x)=\bV_{k+\ell}(\Lambda;x)$ for all $x\in\cX_k$ is  proved for some $\ell\in\dN$.
Then $\xi\in\cX_{k+\ell}$ and $P_{k+\ell}=P_k$, which implies $[P_{k+\ell}(\xi)]v=v$ and $[P_{k+\ell}^\bot(\xi)]v=0$.
Using \eqref{P148-c.5} for $s=k+\ell+1$, we get 
\[\begin{split}
    &[(P_{k+\ell}F^{(k+\ell+1)}P_{k+\ell}+P_{k+\ell}^\bot)(\xi)]v
    =[(P_{k+\ell}F^{(k+\ell+1)}P_{k+\ell})(\xi)]v\\
    &=\sum_{r=1}^k\eta^{(k+\ell+1)}_{r}[P_{k+\ell}(\xi)][(P_{r-1}F^{(r)}P_{r-1})(\xi)][P_{k+\ell}(\xi)]v\\
    &=\sum_{r=1}^k\eta^{(k+\ell+1)}_{r}[P_{k+\ell}(\xi)][(P_{r-1}F^{(r)}P_{r-1})(\xi)]v
    =0,
\end{split}\]
i.\,e., $v\in\ker(P_{k+\ell}F^{(k+\ell+1)}P_{k+\ell}+P_{k+\ell}^\bot)(\xi)$.
From $\xi\in\cX_{k+\ell}$ and \eqref{P148-c.4} we obtain $v\in\bV_{k+\ell+1}(\Lambda;\xi)$.
Since $v\in\bV_{k}(\Lambda;\xi)$ was arbitrary, we have shown $\bV_{k}(\Lambda;\xi)\subseteq\bV_{k+\ell+1}(\Lambda;\xi)$.
Then $\bV_{k}(\Lambda;\xi)=\bV_{k+\ell+1}(\Lambda;\xi)$ by  Lemma~\ref{P148-aP}\ref{P148-aP.b}.
By $\xi\in\cX_{k+\ell}$ and \eqref{P148-c.3}, we have $\det(P_{r-1}F^{(r)}P_{r-1}+P_{r-1}^\bot)(\xi)=0$ for all $r\in\{1,\dotsc,k+\ell\}$.
In particular, $\det(P_{k+\ell-1}F^{(k+\ell)}P_{k+\ell-1}+P_{k+\ell-1}^\bot)(\xi)=0$.
Because of $\xi\in\cX_{k+\ell}$ and \eqref{P148-c.1},  we have  $\xi\in\cX_{k+\ell-1}$.
In view of \eqref{P148-c.4}, then $\bV_{k+\ell}(\Lambda;\xi)\neq\{0\}$.
Hence, $\bV_{k+\ell+1}(\Lambda;\xi)\neq\{0\}$.
By $\xi\in\cX_{k+\ell}$ and \eqref{P148-c.4}, then $\det(P_{k+\ell}F^{(k+\ell+1)}P_{k+\ell}+P_{k+\ell}^\bot)(\xi)=0$, so that $\det(P_{r-1}F^{(r)}P_{r-1}+P_{r-1}^\bot)(\xi)=0$ for all $r\in\{1,\dotsc,k+\ell+1\}$.
In view of \eqref{P148-c.3},  $\xi\in\cX_{k+\ell+1}$.
Thus, since $\xi\in\cX_k$ was arbitrary,   $\cX_k\subseteq\cX_{k+\ell+1}$.
From \eqref{P148-c.1} it follows that $\cX_k=\cX_{k+\ell+1}$.
This completes the induction proof of  \eqref{P148-c.9}.

Using \eqref{1.24}, \eqref{P148-c.1}, and \eqref{P148-c.2} we can infer from \eqref{P148-c.9} that $\cV( \Lambda)=\cX_k$ and $\bV( \Lambda;x)=\bV_{k}(\Lambda;x)$ for all $x\in\cV( \Lambda)$.
Combining this with \eqref{P148-c.1}, \eqref{P148-c.2}, and \eqref{P148-c.9}, we get \eqref{1.25} and \eqref{1.25+}.

Now let $\xi\in\cV( \Lambda)$.
According to \eqref{1.24}, then $\xi\in\cX_k$.
By \eqref{P148-c.3} we have $\xi\in\cX_{k-1}$ and $\det(P_{k-1}F^{(k)}P_{k-1}+P_{k-1}^\bot)(\xi)=0$.
From $\xi\in\cX_{k-1}$ and \eqref{P148-c.4}, then $\bV_{k}(\Lambda;\xi)\neq\{0\}$.
By virtue of \eqref{1.25+},  $\bV( \Lambda;\xi)\neq\{0\}$ follows.
\end{proof} 

\begin{lem}\label{T149<}
$\cW(\Lambda)\subseteq\cV( \Lambda)$ and $\bW(\Lambda;x)\subseteq\bV( \Lambda;x)$ for all $x\in\cW(\Lambda)$.
\end{lem}
\begin{proof}
Let $x\in\cW(\Lambda)$.
By Lemma~\ref{L1459}  there exists a $\nu\in\cM^\mathrm{fa}_\Lambda$ such that $x\in\ats (\nu)$.
Proposition~\ref{P148-bP} yields $\ats (\nu)\subseteq\cV( \Lambda)$.
Consequently, $x\in\cV( \Lambda)$.
Thus we have proved that $\cW(\Lambda)\subseteq\cV( \Lambda)$.

Now let $v\in\bW(\Lambda;x)$.
Because of Lemma~\ref{L1459},  there is $\nu\in\cM^\mathrm{fa}_\Lambda$ such that $v\in\ran (\nu(\{x\}))$.
Since $x\in\cV( \Lambda)$ by $\cW(\Lambda)\subseteq\cV( \Lambda)$, Proposition~\ref{P148-bP} gives $\ran (\nu(\{x\}))\subseteq\bV( \Lambda;x)$.
Thus $v\in\bV( \Lambda;x)$ and $\bW(\Lambda;x)\subseteq\bV( \Lambda;x)$.
\end{proof}

\begin{lem}\label{T149>} $\cV( \Lambda)\subseteq \cW(\Lambda)$ and $\bV( \Lambda;x)\subseteq\bW(\Lambda;x)$ for all $x\in\cV( \Lambda)$.
\end{lem}
\begin{proof}
According to Theorem~\ref{richter}, there exists $\mu\in\cM^\mathrm{fa}_\Lambda$.
Set $\cU:=\cV( \Lambda)$ and $\bU_x:=\bV( \Lambda;x)$ for $x\in \cU$. 
We show that the set $\cU$ and the family $(\bU_x)_{x\in\cU}$ fulfill the assumptions of Corollary~\ref{C131P}:

Lemma~\ref{L0920}  shows that $\cU\in\fX\setminus\{\emptyset\}$.
By Theorem~\ref{P148-c}, there exists $k\in\dN$ such that $\cU=\cV_k(\Lambda)=\cV_{k+1}(\Lambda)$ and $\bU_x=\bV_{k}(\Lambda;x)=\bV_{k+1}(\Lambda;x)$ for all $x\in\cU$.
Clearly, $\bU_x=\bV_{k+1}(\Lambda;x)$ is a linear subspace of $\dC^q$ for  $x\in\cU=\cV_k(\Lambda)$.
By Proposition~\ref{P148-bP}, we have $\ats (\mu)\subseteq\cU$ and $\ran (\mu(\{x\}))\subseteq\bU_x$ for all $x\in\cU$.
Since $\mu\in\cM^\mathrm{fa}_q(\cX,\fX)$, then $\mu(\cX\setminus\mathcal{U})=O$.
Let $\fU:=\cU\cap\fX$ and let  $P\colon\cU\to\cH_q$  be defined by $P(x):=\dP_{\bU_x}$.
Since $\bU_x=\bV_{k}(\Lambda;x)$ for all $x\in\cU=\cV_k(\Lambda)$, then $P=P_k$.
Consequently, Lemma~\ref{L1925}\ref{L1925.d} yields $P\in\cM(\cU,\fU;\cH_q)$.

Let  $F\in\cE$  be such that $P(x)F(x)P(x)\succeq O$ for all $x\in\cU$ and $\Lambda(F)=0$.
Then, by $P=P_k$, $\cU=\cV_k(\Lambda)$, and \eqref{1.22}, we have $F\in\cN_{k+1}(\Lambda)$.
Now let $x\in\cU$.
Then $P(x)=P_k(x)$.
According to \eqref{VVj},  $\bV_{k+1}(\Lambda;x)\subseteq\ker(P_kFP_k+P_k^\bot)(x)$.
Consequently, $\bU_x\subseteq\ker(P(x)F(x)P(x)+[P(x)]^\bot)$.
In view of $\bU_x=\ran (P(x))$, we obtain $(P(x)F(x)P(x)+[P(x)]^\bot)[P(x)]=O$.
Using $[P(x)]^2=P(x)$ and $[P(x)]^\bot P(x)=O$ we get $P(x)F(x)P(x)=O$.
Thus, we have proved that condition~\ref{C131P.I} in Corollary~\ref{C131P} is fulfilled.

Consider arbitrary $\xi\in\cU$ and $v\in\bU_\xi$.
Then $M:=vv^*$ satisfies $M\in\cH_{q,\succeq}$ and $\ran (M)\subseteq\bU_\xi$.
Now Corollary~\ref{C131P} applies, so there exist $\nu\in\cM^\mathrm{fa}_\Lambda$ and $\epsilon>0$ such that $\nu(\{\xi\})\succeq\epsilon M$.
Hence, $v\in\ran (M)\subseteq\ran (\nu(\{\xi\}))$.
By Lemma~\ref{L1459},  $v\in\bW(\Lambda;\xi)$.
Thus, we have shown that $\bV( \Lambda;\xi)\subseteq\bW(\Lambda;\xi)$ for all $\xi\in\cV( \Lambda)$.
Using Theorem ~\ref{P148-c} we can choose $v\in\bU_\xi$ even such that $v\neq 0$.
Then $v\in\ran (\nu(\{\xi\}))$ implies $\nu(\{\xi\})\neq O$, so that $\xi\in\ats (\nu)$.
From Lemma~\ref{L1459},  $\xi\in\cW(\Lambda)$.
This proves that $\cV( \Lambda)\subseteq\cW(\Lambda)$.
\end{proof}

The following theorem is another main result of this paper.

\begin{thm}\label{T149}
Suppose that $\Lambda\neq 0$ is  a moment functional on $\cE$.
Then we have $\cW(\Lambda)=\cV( \Lambda)$ and $\bW(\Lambda;x)=\bV( \Lambda;x)$ for all $x\in\cW(\Lambda)$.
Moreover, $\cW(\Lambda)\in\fX\setminus\{\emptyset\}$.
\end{thm}
\begin{proof}
The first assertions follow by combining Lemmas~\ref{T149<} and~\ref{T149>}.
The equality $\cW(\Lambda)=\cV( \Lambda)$ combined with Lemma~\ref{L0920} yields $\cW(\Lambda)\in\fX\setminus\{\emptyset\}$.
\end{proof}

\section{Some examples}\label{someexamples}
In this section we determine for three simple examples the notions and quantities introduced in Definitions~\ref{D133},~\ref{D147}, and~\ref{cordef} explicitly. 

\begin{exm}\label{exa2}
Throughout this example, we abbreviate
\begin{align}
    A_1&:=
    \begin{pmatrix}
    1&0&0\\
    0&0&0\\
    0&0&0
    \end{pmatrix},&
    A_2&:=
    \begin{pmatrix}
    0&0&0\\
    0&1&0\\
    0&0&0
    \end{pmatrix},&
    A_3&:=
    \begin{pmatrix}
    0&0&0\\
    0&0&0\\
    0&0&1
    \end{pmatrix},\label{a123}\\
    B_1&:=
    \begin{pmatrix}
    1&0&0\\
    0&1&0\\
    0&0&0
    \end{pmatrix},&
    M&:=
    \begin{pmatrix}
    1&1&1\\
    1&1&1\\
    1&1&1
    \end{pmatrix}.\label{mmatrix}
\end{align}
Let $\cX:=\dR$, $q=3$, and $\cE:=\Lin\{F_1,F_2,F_3\}$, where
\[
    F_1(x):=A_{1},\quad
    F_2(x):=xA_{2},\quad
    F_3(x):=x^2(x-1)^2A_{3},
\]
and the matrices $A_1, A_2, A_3$ are defined by \eqref{a123}.
Set $\mu:=\delta_0M$ and $\Lambda:=\Lambda^\mu$.

Then $\ker\Lambda=\Lin \{F_2,F_3\}$ and $\cE_\succeq=\cone \{F_1,F_3\}$, so that
\begin{align*}
    \cN_1(\Lambda)&=\ker\Lambda\cap\cE_\succeq=\cone \{F_3\},\\
    \cV_1(\Lambda)&=\{x\in\dR\colon\det F_3(x)=0\}=\dR,\\
    \bV_1(\Lambda;x)&=\ker F_3(x)=
    \begin{cases}
        \dC^3;&x\in\{0,1\}\\
        \Lin \{e_1,e_2\};&x\in\dR\setminus\{0,1\}
    \end{cases},\\
    P_1(x)&=
    \begin{cases}
        I_3;&x\in\{0,1\}\\
        B_1;&x\in\dR\setminus\{0,1\}
    \end{cases},\quad
    P_1^\bot(x)=
    \begin{cases}
        O;&x\in\{0,1\}\\
        A_3;&x\in\dR\setminus\{0,1\}
    \end{cases}.
\end{align*}
We have
\[
    P_1F_1P_1=F_1,\quad
    P_1F_2P_1=F_2,\quad
    P_1F_3P_1=O
\]
and hence
\begin{align*}
    \cN_2(\Lambda)&=\{F\in\ker\Lambda\colon (P_1FP_1)(x)\succeq O\text{ for all }x\in\cV_1(\Lambda)\}=\cone \{F_3\},\\
    \cV_2(\Lambda)&=\{x\in\dR\colon\det(P_1F_3P_1+P_1^\bot)(x)=0\}=\dR,\\
    \bV_2(\Lambda;x)&=\ker(P_1F_3P_1+P_1^\bot)(x)=
    \begin{cases}
        \dC^3;&x\in\{0,1\}\\
        \Lin \{e_1,e_2\};&x\in\dR\setminus\{0,1\}
    \end{cases}.
\end{align*}
Consequently,
\begin{align*}
    \cW(\Lambda)&=\cV( \Lambda)=\cV_1(\Lambda)=\cV_{1+\ell}(\Lambda)=\dR,\\
    \bW(\Lambda;x)&=\bV( \Lambda;x)=\bV_1(\Lambda;x)=\bV_{1+\ell}(\Lambda;x)
\end{align*}
for all $\ell\in \dN_0$, $x\in \dR$. That is, we have $k=1$ in Theorem~\ref{P148-c}. 
For instance,
\begin{align*}
    \nu_0&:=\delta_0I_3\in\cM_\Lambda&\text{with }\ran (\nu_0(\{0\}))&=\dC^3,\\
    \nu_1&:=\delta_{-1}A_2+\delta_1I_3\in\cM_\Lambda&\text{with }\ran (\nu_1(\{1\}))&=\dC^3,\\
    \sigma&:=\delta_{-\xi}A_2+\delta_\xi B_1\in\cM_\Lambda&\text{with }\ran (\sigma(\{\xi\}))&=\Lin \{e_1,e_2\}
\end{align*}
for arbitrary $\xi\in\dR$.
\end{exm}

\begin{exm}\label{exa3}
Let $A_1$, $A_2$, $A_3$ and $M$ be the matrices defined by \eqref{a123} and \eqref{mmatrix}.
In addition we need the matrices
\[
    B_1:=
    \begin{pmatrix}
    1&0&0\\
    0&1&0\\
    0&0&0
    \end{pmatrix},\quad
    B_2:=
    \begin{pmatrix}
    0&0&0\\
    0&1&0\\
    0&0&1
    \end{pmatrix}.
\]
Let $\cX=\dR^3$, $q=3$, and $\cE:=\Lin \{F_1, F_2, F_3\}$, where
\[
    F_1(x):=A_1,\quad
    F_2(x):=x^2(x-2)^2A_2+x A_3,\quad
    F_3(x):=x^2(x-1)^2A_3.
\]
Set $\mu:=\delta_0M$ and $\Lambda:=\Lambda^\mu$.
Then $\ker\Lambda=\Lin \{F_2,F_3\}$ and $\cE_\succeq=\cone \{F_1,F_3\}$, so that
\begin{align*}
    \cN_1(\Lambda)&=\ker\Lambda\cap\cE_\succeq=\cone \{F_3\},\\
    \cV_1(\Lambda)&=\{x\in\dR\colon\det F_3(x)=0\}=\dR,\\
    \bV_1(\Lambda;x)&=\ker F_3(x)=
    \begin{cases}
        \dC^3;&x\in\{0,1\}\\
        \Lin \{e_1,e_2\};&x\in\dR\setminus\{0,1\}
    \end{cases},\\
    P_1(x)&=
    \begin{cases}
        I_3;&x\in\{0,1\}\\
        B_1;&x\in\dR\setminus\{0,1\}
    \end{cases},\quad
    P_1^\bot(x)=
    \begin{cases}
        O;&x\in\{0,1\}\\
        A_3;&x\in\dR\setminus\{0,1\}
    \end{cases}.
\end{align*}
We have
\[
    P_1F_1P_1=F_1,\quad
    (P_1F_2P_1)(x)=
    \begin{cases}
        B_2;&x=1\\
        x^2(x-2)^2A_2&x\in\dR\setminus\{1\}
    \end{cases},\quad
    P_1F_3P_1=O.
\]
Hence
\begin{align*}
    \cN_2(\Lambda)&=\{F\in\ker\Lambda\colon (P_1FP_1)(x)\succeq O\text{ for all }x\in\cV_1(\Lambda)\}=\cone \{F_2,F_3\},\\
    \cV_2(\Lambda) &=\{x\in\dR\colon\det(P_1F_2P_1+P_1^\bot)(x)=0\}=\dR,\\
    \bV_2(\Lambda;x)&=\ker(P_1F_2P_1+P_1^\bot)(x)=
    \begin{cases}
        \dC^3;&x=0\\
        \Lin \{e_1,e_2\};&x=2\\
        \Lin \{e_1\};&x\in\dR\setminus\{0,2\}
    \end{cases},\\
    P_2(x)&=
    \begin{cases}
        I_3;&x=0\\
        B_1;&x=2\\
        A_1;&x\in\dR\setminus\{0,2\}
    \end{cases},\quad
    P_2^\bot(x)=
    \begin{cases}
        O;&x=0\\
        A_3;&x=2\\
        B_2;&x\in\dR\setminus\{0,2\}
    \end{cases},
\end{align*}
and
\begin{gather*}
    P_2F_1P_2=F_1,\quad
    P_2F_2P_2=
    \begin{cases}
    O;&x\in \dR\setminus\{0\}\\
    F_2;&x=0
    \end{cases},\\
    P_2F_3P_2=\begin{cases}
    O;&x\in \dR\setminus\{0\}\\
    F_3;&x=0
    \end{cases}.
\end{gather*}
Using these formulas we compute $\cN_3(\Lambda)=\cN_2(\Lambda)$, $\cV_3(\Lambda)=\cV_2(\Lambda)$ and $\bV_3(\Lambda;x)= \bV_2(\Lambda;x)$ for  $x\in \dR$. Therefore,
\begin{align*}
    \cW(\Lambda)&=\cV( \Lambda)=\cV_2(\Lambda)=\cV_{2+\ell}(\Lambda)=\dR,\\
    \bW(\Lambda;x)&=\bV( \Lambda;x)=\bV_2(\Lambda;x)=\bV_{2+\ell}(\Lambda;x)
\end{align*}
for all $\ell\in \dN_0$, $x\in \dR$. Thus, we have $k=2$ in Theorem~\ref{P148-c}. 

Moreover, for instance, $\nu_0:=\delta_0I_3\in\cM_\Lambda$ with $\ran (\nu_0(\{0\}))=\dC^3$, $\nu_2:=\delta_2B_1\in\cM_\Lambda$ with $\ran (\nu_2(\{2\}))=\Lin \{e_1,e_2\}$, and $\sigma:=\delta_\xi A_1\in\cM_\Lambda$ with $\ran (\sigma(\{\xi\}))=\Lin \{e_1\}$ for arbitrary $\xi\in\dR$.
\end{exm}

\begin{exm}\label{exa1}
Let $\cX:=\dR$, $q=2$, and $\cE:=\Lin\{F_1,F_2,F_3\}$, where
\[
    F_1(x):=A_{1},\quad
    F_2(x):=xA_{2},\quad
    F_3(x):=x^2(x-1)^2A_{3},
\]
and
\[
    A_{1}:=
    \begin{pmatrix}
    1&0\\
    0&0
    \end{pmatrix},\qquad
    A_{2}:=
    \begin{pmatrix}
    0&0\\
    0&1
    \end{pmatrix},\qquad
    A_{3}:=
    \begin{pmatrix}
    1&1\\
    1&1
    \end{pmatrix}.  
\]

Set $\mu:=\delta_0I_2$ and  $\Lambda:=\Lambda^\mu$.
Then  $\ker\Lambda=\Lin\{F_2,F_3\}$, $\cE_\succeq=\cone \{F_1,F_3\}$, and
\begin{align*}
    \cN_1(\Lambda)&=\ker\Lambda\cap\cE_\succeq=\cone \{F_3\},\\
    \cV_1(\Lambda)&=\{x\in\dR\colon\det F_3(x)=0\}=\dR,\\
    \bV_1(\Lambda;x)&=\ker F_3(x)=
    \begin{cases}
        \dC^2;&x\in\{0,1\}\\
        \dC\cdot(-1,1)^\trn ;&x\in\dR\setminus\{0,1\}
    \end{cases},\\
    P_1(x)&=
    \begin{cases}
        I_2;&x\in\{0,1\}\\
        B;&x\in\dR\setminus\{0,1\},
    \end{cases}, \quad\text{where}\quad  B
    :=\frac{1}{2}\begin{pmatrix}1&-1\\-1&1\end{pmatrix}.
\end{align*}
In this case, we derive $\cN_2(\Lambda)=\cN_1(\Lambda)$, $\cV_2(\Lambda)=\cV_1(\Lambda)$, and $\bV_2(\Lambda;x)=\bV_1(\Lambda;x)$. 
Consequently,
\begin{align*}
    \cW(\Lambda)&=\cV( \Lambda)=\cV_1( \Lambda)=\cV_{1+\ell}(\Lambda)=\dR,\\
    \bW(\Lambda;x)&=\bV(\Lambda;x)=\bV_1(\Lambda;x)=\bV_{1+\ell}(\Lambda;x)
\end{align*}
for all $\ell\in \dN_0$ and $x\in \dR$.
That is,  $k=1$ in Theorem~\ref{P148-c}. 
For instance,
\begin{align*}
    \nu_0&:=\delta_0I_2=\mu\in\cM_\Lambda&\text{with }\ran (\nu_0(\{0\}))&=\dC^2,\\
    \nu_1&:=\delta_{-1}B+\frac{1}{2}\delta_1I_2\in\cM_\Lambda&\text{with }\ran (\nu_1(\{1\}))&=\dC^2,\\
    \sigma&:=\delta_{-\xi}B+\delta_\xi B\in\cM_\Lambda&\text{with }\ran (\sigma(\{\xi\}))&=\dC\cdot (-1,1)^\trn ,
\end{align*}
for any $\xi\in\dR$.
\end{exm}


\begin{thebibliography}{FKM21}
    
\bibitem[And70]{Ando}%{MR290157}
And\^{o}, T.:
\newblock {\em Truncated moment problems for operators.}
\newblock Acta Sci. Math. (Szeged) 31 (1970), 319--334.

\bibitem[AT06]{AdamyanT}%{MR2215856}
Adamyan, V. M.; Tkachenko, I. M.:
\newblock {\em General solution of the {S}tieltjes truncated matrix moment problem.} Operator theory and indefinite inner product spaces, 1--22,
\newblock Oper. Theory Adv. Appl., 163, Birkh\"{a}user, Basel, 2006.

\bibitem[BIG03]{MR1987382}
Ben-Israel, A.; Greville, T. N. E.:
\newblock {\em Generalized inverses.}
\newblock Theory and applications. Second edition. CMS Books in Mathematics 15.
\newblock Springer-Verlag, New York, 2003.

%\bibitem[Bol96]{Bolotnikov}%{MR1395706}
%Bolotnikov, V. A.:
%\newblock {\em On degenerate {H}amburger moment problem and extensions of nonnegative {H}ankel block matrices.}
%\newblock Integral Equations Operator Theory 25 (1996), no.~3, 253--276.

\bibitem[BW11]{BakonyiW}%{MR2807419}
Bakonyi, M.; Woerdeman, H. J.:
\newblock {\em Matrix completions, moments, and sums of {H}ermitian squares.}
\newblock Princeton University Press, Princeton, NJ, 2011.

\bibitem[CF96]{MR1303090}
Curto, R. E.; Fialkow, L. A.: 
\newblock {\em Solution of the truncated complex moment problem for flat data.}
\newblock Mem. Amer. Math. Soc. 119 (1996), no.~568, x+52 pp.

%\bibitem[CF98]{MR1445490}
%Curto, R. E.; Fialkow, L. A.:
%\newblock {\em Flat extensions of positive moment matrices: recursively generated relations.}
%\newblock Mem. Amer. Math. Soc. 136 (1998), no. 648, x+56 pp.

\bibitem[CF05]{MR2168867}
Curto, R. E.; Fialkow, L. A.:
\newblock {\em Truncated {$K$}-moment problems in several variables.}
\newblock J. Operator Theory 54 (2005), no.~1, 189--226.

%\bibitem[CH98]{MR1624548}
%Chen, G.-n.; Hu, Y.-j.:
%\newblock {\em The truncated {H}amburger matrix moment problems in the nondegenerate and degenerate cases, and matrix continued fractions.}
%\newblock Linear Algebra Appl. 277 (1998), no.~1-3, 199--236.

\bibitem[CZ13]{MR3011271}
Cimpri\v{c}, J.; Zalar, A.:
\newblock {\em Moment problems for operator polynomials.}
\newblock J. Math. Anal. Appl. 401 (2013), no. 1, 307--316.

%\bibitem[dD19]{Dio}%{MR3973903}
%di Dio, P. J.:
%\newblock {\em The multidimensional truncated moment problem: {G}aussian and log-normal mixtures, their {C}arath\'{e}odory numbers, and set of atoms.}
%\newblock Proc. Amer. Math. Soc. 147 (2019), no.~7, 3021--3038.

\bibitem[dDK21]{MR4263684}
di Dio, P. J.; Kummer, M.:
\newblock {\em The multidimensional truncated moment problem: {C}arath\'{e}odory numbers from {H}ilbert functions.}
\newblock Math. Ann. 380 (2021), no.~1-2, 267--291.

\bibitem[dDS18a]{DioS1}%{MR3782989}
di Dio, P. J.; Schm\"udgen, K.:
\newblock {\em The multidimensional truncated moment problem: atoms, determinacy, and core variety.}
\newblock J. Funct. Anal. 274 (2018), no. 11, 3124--3148.

\bibitem[dDS18b]{DioS2}%{MR3765506}
di Dio, P. J.; Schm\"{u}dgen, K.:
\newblock {\em The multidimensional truncated moment problem: {C}arath\'{e}odory numbers.}
\newblock J. Math. Anal. Appl. 461 (2018), no.~2, 1606--1638.

\bibitem[dDS22]{MR4379909}
di Dio, P. J.; Schm\"udgen, K.:
\newblock {\em The multidimensional truncated moment problem: the moment cone.}
\newblock J. Math. Anal. Appl. 511 (2022), no. 1, Paper No. 126066.

%\bibitem[DLR97]{MR1471052}
%Duran, A. J.; Lopez-Rodriguez, P.:
%\newblock {\em Density questions for the truncated matrix moment problem.}
%\newblock Canad. J. Math. 49 (1997), no.~4, 708--721.

\bibitem[Dym89]{Dym}%{MR1018213}
Dym, H.:
\newblock {\em On {H}ermitian block {H}ankel matrices, matrix polynomials, the {H}amburger moment problem, interpolation and maximum entropy.}
\newblock Integral Equations Operator Theory 12 (1989), no.~6, 757--812.

%\bibitem[Dyu04]{MR2053150}
%Dyukarev, Yu. M.:
%\newblock {\em Indeterminacy criteria for the {S}tieltjes matrix moment problem.} (Russian. Russian summary)
%\newblock Mat. Zametki 75 (2004), no.~1, 71--88;
%\newblock translation in Math. Notes 75 (2004), no.~1-2, 66--82.

\bibitem[Fia17]{MR3688460}
Fialkow, L. A.:
\newblock {\em The core variety of a multisequence in the truncated moment problem.}
\newblock J. Math. Anal. Appl. 456 (2017), no.~2, 946--969.

%\bibitem[FKL03]{MR1999137}
%Fritzsche, B.; Kirstein, B.; Lasarow, A.:
%\newblock {\em On a moment problem for rational matrix-valued functions.}
%\newblock Linear Algebra Appl. 372 (2003), 1--31.

\bibitem[FKM21]{FritzscheKM}%{MR4206128}
Fritzsche, B.; Kirstein, B.; M\"{a}dler, C.:
\newblock {\em A {S}chur-{N}evanlinna type algorithm for the truncated matricial {H}ausdorff moment problem.}
\newblock Complex Anal. Oper. Theory 15 (2021), no.~2, Paper No.~25, 129~pp.

\bibitem[K14]{Kimsey}%{MR3146815}
Kimsey, D. P.:
\newblock {\em An operator-valued generalization of {T}chakaloff's theorem.}
\newblock J. Funct. Anal. 266 (2014), no.~3, 1170--1184.


\bibitem[KT22]{KimseyT}%{MR4443113}
Kimsey, D. P.; Trachana, M.:
\newblock {\em On a solution of the multidimensional truncated matrix-valued moment problem.}
\newblock Milan J. Math. 90 (2022), no.~1, 17--101.

\bibitem[KW13]{KimseyW}%{MR3074378}
Kimsey, D. P.; Woerdeman, H. J.:
\newblock {\em The truncated matrix-valued {$K$}-moment problem on {$\dR^d$}, {$\dC^d$}, and {$\mathbb{T}^d$}.}
\newblock Trans. Amer. Math. Soc. 365 (2013), no.~10, 5393--5430.

\bibitem[Kre49]{Krein}%{MR34964}
Kre\u{\i}n, M. G.:
\newblock {\em Infinite {$J$}-matrices and a matrix-moment problem. (Russian)}
\newblock Doklady Akad. Nauk SSSR (N.S.) 69, (1949). 125--128.

\bibitem[Roc70]{MR0274683}
Rockafellar, R. T.:
\newblock {\em Convex analysis.}
\newblock Princeton Mathematical Series, No.~28 Princeton University Press, Princeton, N.J. 1970.

\bibitem[RS18]{RienerS}%{MR3748596}
Riener, C.; Schweighofer, M.:
\newblock {\em Optimization approaches to quadrature: new characterizations of {G}aussian quadrature on the line and quadrature with few nodes on plane algebraic curves, on the plane and in higher dimensions.}
\newblock J. Complexity 45 (2018), 22--54.

\bibitem[Sch17]{Schm17}%{MR3729411}
Schm\"{u}dgen, K.:
\newblock {\em The moment problem.}
\newblock Graduate Texts in Mathematics, 277. Springer, Cham, 2017.

%\bibitem[Sim06]{MR2238039}
%Simonov, K. K.:
%\newblock {\em Strong matrix moment problem of {H}amburger.}
%\newblock Methods Funct. Anal. Topology 12 (2006), no.~2, 183--196.

\bibitem[TL20]{Le}%{MR4207169}
Trinh Le, Cong:
\newblock {\em Tracial moment problems on hypercubes.}
\newblock Oper. Matrices 14 (2020), no.~4, 1015--1027.


\bibitem[Vas98]{vasilescu}
Vasilescu, F.-H.:
\newblock {\em Moment problems for multi-sequences of operators.}
\newblock J. Math. Anal. Appl. 219 (1998), 246--259.

\bibitem[Zag10]{Zag}%{MR2743592}
Zagorodnyuk, S. M.:
\newblock {\em A description of all solutions of the matrix {H}amburger moment problem in a general case.}
\newblock Methods Funct. Anal. Topology 16 (2010), no.~3, 271--288.

\end{thebibliography}
\end{document}